\documentclass{article}[12pt]

\usepackage{amsmath, amsthm, amssymb}
\usepackage{enumerate}
\usepackage{authblk}
\usepackage[ruled,vlined,linesnumbered,resetcount,algosection]{algorithm2e}
\usepackage{color}
\usepackage{tikz}
\pgfdeclarelayer{edgelayer}
\pgfdeclarelayer{nodelayer}
\pgfsetlayers{edgelayer,nodelayer,main}
\usetikzlibrary {decorations.pathmorphing, decorations.pathreplacing, decorations.shapes, patterns}
\usepackage[a4paper, total={6in, 8in}]{geometry}

\newtheorem{theorem}{Theorem}[section]
\newtheorem{definition}[theorem]{Definition}
\newtheorem{remark}[theorem]{Remark}
\newtheorem{lemma}[theorem]{Lemma}
\newtheorem{proposition}[theorem]{Proposition}
\newtheorem{corollary}[theorem]{Corollary}
\newtheorem{problem}[theorem]{Problem}

\theoremstyle{definition}
\newtheorem{case}{Case}
\newtheorem{subcase}{Subcase}[case]

\DeclareMathOperator{\diam}{diam}
\DeclareMathOperator{\rad}{rad}
\DeclareMathOperator{\e}{ecc}
\DeclareMathOperator{\cen}{center}
\DeclareMathOperator{\bor}{bor}
\DeclareMathOperator{\reach}{\mathfrak{R}}

\newcommand{\svSpan}[1]{\sigma^{\boxtimes}_V(#1)}
\newcommand{\seSpan}[1]{\sigma^{\boxtimes}_E(#1)}

\begin{document}
\title{The strong vertex span of trees}
\author{
Mateja Grašič$^{a,b}$\thanks{mateja.grasic@um.si}\and Chris Mouron$^{c}$\thanks{mouronc@rhodes.edu} \and  Andrej Taranenko$^{a,b}$\thanks{andrej.taranenko@um.si}
}

\date{\today}

\maketitle

\begin{center}
$^a$ Faculty of Natural Sciences and Mathematics, University of Maribor, Slovenia \\
$^b$ Institute of Mathematics, Physics and Mechanics, Ljubljana, Slovenia\\
$^c$  Department of Mathematics and Statistics, Rhodes College, Memphis, TN, 38112 USA
\end{center}

\begin{abstract}
The strong vertex (edge) span of a given graph $G$ is the maximum distance that two players can maintain at all times while visiting all vertices (edges) of $G$ and moving either to an adjacent vertex or staying in the current position independently of each other. We introduce the notions of switching walks and triod size of a tree, which are used to determine the strong vertex and the strong edge span of an arbitrary tree. The obtained results are used in an algorithm that computes the strong vertex (edge) span of the input tree in linear time.


\end{abstract}

\section{Introduction and basic definitions}
Inspired by Lelek's notion of the span of a continuum introduced in 1964 \cite{Lelek}, Banič and Taranenko \cite{BaTa23} translated this idea to graph theory and  defined spans of a graph. In \cite{BaTa23}, several variants of spans were presented. In this paper we focus on one of these, namely the strong vertex span of a given graph. The strong vertex span of a graph can be explained in the language of games on graphs as follows: two players, say Alice and Bob, are moving through a given graph and wish to traverse all vertices of the graph and are allowed to move independently of each other either to an adjacent vertex or stay at the same vertex. At each point in time, the distance between the two players is measured. The minimum distance obtained over all moments in time is the \emph{safety distance} Alice and Bob were able to maintain. The question asked is, given a connected graph what is the maximum possible safety distance the two players are able to maintain at each point in time. 

In the seminal paper by Banič and Taranenko \cite{BaTa23},  notions of spans of a graph were formally introduced. All variants of spans were characterised with respect to subgraphs of graph products and it was shown that the value of a chosen span can be obtained in polynomial time (a polynomial of degree 4 with respect to the number of vertices of the given graph). Also, 0-span graphs were characterised for each span variant. Erceg et al. \cite{Erceg23} soon after published further results on spans of a graph, where the relation between different vertex span variants was studied, and spans for specific families of graphs were determined. In particular, related to this paper, the strong vertex span of perfect binary trees was determined. Continuing the work on spans, Šubašić and Vojković \cite{SuVo24} determined the values of all variants of vertex spans for multilayered graphs and their sub-classes, i.e. multilayered cycles and multilayered paths. 

In this paper we focus our research on the strong vertex span of trees and proceed as follows:
\begin{enumerate}
    \item In the remainder of this section we present basic definitions and the results needed throughout the paper.
    \item In Section \ref{sec:svTrees} we introduce the notion of the triod size of a tree and use it to determine the strong vertex span of an arbitrary tree.
    \item We use the previously obtained results in Section \ref{sec:treesAlg} to present a linear time algorithm that computes the strong vertex span of an arbitrary tree.
\end{enumerate}     

Let $G$ be a connected graph with the vertex set $V(G)$, the edge set $E(G)$ and $v$ a vertex of $G$. The \emph{eccentricity} of the vertex $v$, denoted by $\e(v)$, is the maximum distance from $v$ to any vertex of $G$. That is, $\e(v)=\max\{d_G(v,u) \mid u \in V(G)\}.$ The \emph{radius} of $G$ is $\rad(G)=\min\{\e(v) \mid v \in V(G)\}$ and the \emph{diameter} of $G$ is $\diam(G)=\max\{\e(v) \mid v \in V(G)\}.$ 

Let $G$ and $H$ be any graphs. A function $f: V(G) \rightarrow V(H)$ is a \emph{weak homomorphism from $G$ to $H$} if  for all $u,v\in V(G)$, $uv\in E(G)$ implies $f(u) f(v) \in E(H)$ or $f(u)=f(v)$. We will use the more common notation $f:G\rightarrow H$ to say that $f: V(G) \rightarrow V(H)$ is a weak homomorphism. A weak homomorphism $f: G \rightarrow H$ is \emph{surjective} if $f(V(G)) = V(H)$. Let $f,g: G \rightarrow H$ be weak homomorphisms. The \emph{distance from $f$ to $g$} is defined as $m_G(f,g) = \min \{ d_H(f(u), g(u))  \mid u \in V(G) \}$.  If $f,g: G \rightarrow H$ are surjective weak homomorphisms and $G$ is connected, then $m_G(f,g)\leq \rad(H)$ \cite{BaTa23}.

We are now ready to state the definitions of the strong vertex span as defined in \cite{BaTa23}. We follow the notations as presented in the original paper.

\begin{definition}[\cite{BaTa23}]\label{def:strongSpans}
Let $H$ be a connected graph. Define the \emph{strong vertex span} of the graph $H$, denoted by $\svSpan{H}$, as
\[
     \svSpan{H} = \max \{ m_P(f,g)  \mid f,g: P \rightarrow H \text{ are surjective weak homomorphisms and $P$ is a path} \}. 
\]
\end{definition}

The usage of (surjective) weak homomorphisms can be explained as follows. Walks on graphs parameterised by time can be presented by paths, where two adjacent vertices represent consecutive points in time. A mapping from a path to the given graph defined by a weak homomorphism represents one walk on the given graph. A weak homomorphism can map two adjacent vertices to either the same vertex, meaning that the player did not move at this point in time, or to two adjacent vertices on the graph, this means the player moved to an adjacent vertex. The condition that we consider only surjective weak homomorphisms implies that the walk presented by the weak homomorphism is a walk through all vertices of the graph. So two surjective weak  homomorphisms $f$ and $g$ in Definition \ref{def:strongSpans} represent a pair of walks through all vertices of the graph, one for each player. Clearly, the value $m_P(f,g)$ equals to the minimum of the distances between the players over all points in time. Considering all possible valid walks, we obtain the value of the strong vertex span. Throughout the paper we will use both the language of walks on graphs or the formal notions of weak homomorphisms to study the strong vertex span of a tree. Both are necessary, as it turns out some things are more easily proved in one language, others in the other.

Let $G$ be a connected graph and $v\in V(G)$ be a vertex. We define a {\it component} $C(v)$ of $G-\{v\}$ to be a maximal connected subgraph of $G-\{v\}$. Then we define the {\it border} of $C(v)$, denoted by  $\bor(C(v))$, to be the vertices of $C(v)$ that are adjacent to $v$ in $G$, and the {\it closure} of $C(v)$, denoted by $\overline{C(v)}$, to be the subgraph of $G$ induced by the vertex set $V(C(v))\cup\{v\}$. 

\begin{remark}
For the reasons of brevity we in some parts abuse the notation and write for a graph $G$ that $v\in G$, where we mean $v\in V(G)$.    
\end{remark}
 
The {\it reach} of $C(v)$ is the eccentricity of $v$ in $\overline{C(v)}$, that is $\reach(C(v))=\max\{d(v,u) \mid u\in \overline{C(v)}\}$. 

\section{Strong vertex span of a tree}\label{sec:svTrees}

    Throughout the paper if not stated otherwise we use the following notation of vertices of a path. If $P$ is a path, then we denote the vertices of $P$ with the numbers $\{1, 2, \ldots\, |V(P)|\}$, where for each $i\in \{1, 2, \ldots\, |V(P)|-1\}$, $i(i+1) \in E(P)$. Any comparisons of vertices of a path with relations $\leq, \geq, <, >$ and therefore computations of $\min/\max$ as well as arithmetic operations are done on the labels of the vertices.

\begin{definition}\label{def:switch}
  Let $P$ be a path and $A:P\longrightarrow G$ and  $B:P\longrightarrow G$ be weak homomorphisms. We say that \emph{$B$ switches with $A$ at vertex $v\in V(G)$} if there exist $i,j \in V(P)$, such that $i<j$ and distinct components $C_{\alpha}(v), C_{\beta}(v), C_{\gamma}(v)$ such that $A(i)\in C_{\alpha}(v)$, $B(i)\in C_{\beta}(v)$, $A(j)\in C_{\gamma}(v)$, $B(j)=v$ and $B(t)\not=v$ for all $i\leq t< j$, therefore $B(t) \in C_{\beta}(v)$ for all $i\leq t < j$.  
\end{definition}

\begin{remark}
The notion of a switch is not a symmetric idea, $B$ switching with $A$ at $v$ is not the same as $A$ switching with $B$ at $v$. 
\end{remark}

To determine the strong vertex span of any tree we need some auxiliary results. In Lemma \ref{leg} we use notations from Definition \ref{def:switch}.

\begin{lemma}\label{leg} Let  $A:P\longrightarrow G$ and  $B:P\longrightarrow G$ be weak homomorphisms such that $B$  switches with $A$  at some vertex $v$. Then \[\min\{d(A(t),B(t))\mid t\in V(P)\} \leq \min\{\reach(C_{\beta}(v)),\reach(C_{\gamma}(v))\}.\]
\end{lemma}
\begin{proof} Let $i,j, C_{\beta}(v)$ and $C_{\gamma}(v)$ be as in Definition \ref{def:switch}. Since $A(j)\in C_{\gamma}(v)$, there exists $i < t < j$ such that $A(t)=v$. Since $B(t)\in C_{\beta}(v)$, it follows that $d(A(t),B(t))\leq \reach(C_{\beta}(v))$. Also, since $A(j)\in C_{\gamma}(V)$, it follows that $d(A(j),B(j))= d(A(j),v)\leq  \reach(C_{\gamma}(v))$.
\end{proof}

\begin{lemma}\label{switch}
	Let $T$ be a tree and $A:P\longrightarrow T$ and $B:P\longrightarrow T$ weak homomorphisms with $d(A(t),B(t))\geq 2$ for all $t \in V(P)$. If there exist $i, j \in V(P)$, where $i<j$, such that
	\begin{enumerate}
        \item For all  $i \leq t < j$, $B(t)$ is not in the path between $A(i)$ and $A(t)$
        \item $B(j)$ is in the path between $A(i)$ and $A(j)$, 
	\end{enumerate}
	then $B$ switches with $A$ at the vertex $B(j)$. 
\end{lemma}

\begin{proof}
    We will use the following notations: $v=B(j)$, $C_\alpha(v)$ is the component of $T - \{v\}$ that contains $A(i)$, $C_\beta(v)$ is the component of $T - \{v\}$ that contains $B(j-1)$ and $C_\gamma(v)$ is the component of $T - \{v\}$ that contains $A(j)$.

    \medskip
    
    {\bf Claim:} $C_{\alpha}(v)$, $C_{\beta}(v)$ and $C_{\gamma}(v)$ are all distinct components of $T-\{v\}$.

    \medskip

    Note, since $d(A(j), B(j))\geq 2$, then $A({j-1})$ belongs to $C_{\gamma}(v)$. Working towards a contradiction suppose that $B({j-1}) \in C_{\alpha}(v)$. This implies that $B({j-1})\in \bor(C_{\alpha}(v))$ since $B({j-1})$ is adjacent to $v$. So, $B({j-1})$ is in the path from $A(i)$ to $A({j-1})$ which contradicts condition 1. Hence, $C_{\alpha}(v) \not= C_{\beta}(v)$. 
    Clearly, $A(j)\not\in C_{\alpha}(v)$, otherwise, the path from $A(i)$ to $A(j)$ would be in $C_{\alpha}(v)$ and hence not contain $v$, contradicting condition 2, therefore $C_{\alpha}(v) \not= C_{\gamma}(v)$.
    Finally, if $A(j)\in C_{\beta}(v)$, then it follows from $d(A(j),B(j))\geq 2$ that $A({j-1})\in C_{\beta}(v)$. Also, since $B({j-1})\in C_{\beta}(v)$ and it is adjacent to $v$, it follows that  $B({j-1})\in \bor(C_{\beta}(v))$. So, $B({j-1})$ is in the path from $A(i)$ to $A({j-1})$ which contradicts condition 1. Thus, $C_{\beta}(v) \not= C_{\gamma}(v)$. 

    \medskip

    Let $t' = \max\{t \mid i \leq t< j \text{ and } A(t) \in C_\alpha(v)\}$. Clearly, $t'$ exists since $A(i)\in C_\alpha(v)$ and $A(j)\not \in C_\alpha(v)$. 

    \medskip
    
    {\bf Claim:}  For all $t' \leq t < j$ it holds true that $B(t)\not = v$.

    \medskip

    From the choice of $t'$ we have that $A(t) \not \in C_\alpha(v)$ for any $t'+1 \leq t < j$ and $A(t')\in \bor(C_\alpha(v))$, therefore for all $t' \leq t < j$ any path from $A(i)$ to $A(t)$ contains $v$. Using condition 1 and that $d(A(t'),B(t'))\geq 2$ we have that $B(t)\not = v$ for any $t'\leq t < j$. 

    \medskip

    {\bf Claim:}  $B(t') \in C_{\beta}(v)$.

    \medskip
    
    Working towards a contradiction suppose, that $B(t')\not \in C_\beta(v)$. Clearly, the path from $B(t') \not \in C_\beta(v)$ to $B({j-1})\in C_\beta(v)$ contains $v$ and since $t' < j$, there exists $t' \leq t < j$ such that $B(t) = v$, a contradiction to the previous claim. 
    
    \medskip

    We now have, there exist $t'$ and $j$ with $t'<j$, such that  if $v = B(j)$, then $C_{\alpha}(v)$, $C_{\beta}(v)$ and $C_{\gamma}(v)$ are distinct components, $A(t') \in C_{\alpha}(v)$, $B(t') \in C_{\beta}(v)$, $A(j) \in C_{\gamma}(v)$ and for all $t' \leq t < j$ it holds true that $B(t)\not = v$. By Definition \ref{def:switch}, $B$ switches with $A$ at $v$.
\end{proof}

Throughout the rest of the paper we will also need the notion of the triod size of a vertex and the triod size of a tree.

\begin{definition}\label{def:triod-max-eta}
Let $T$ be a tree and $v$ be a vertex of $T$. Let $C_{1}(v), C_{2}(v),...,C_{deg(v)}(v)$ be the components of $T-\{v\}$ denoted such that $\reach(C_i(v))\geq \reach(C_{i+1}(v))$ for each $i \in \{1, 2, \ldots, \deg(v)-1\}$. Define \emph{triod size of $v$} as \[\eta(v)= \begin{cases}
   \reach(C_3(v)), & \text{if } \deg(v)\geq 3 \\
   0, & \text{if }\deg(v)\in\{0,1,2\}
\end{cases} \]
and \emph{triod size of $T$} as \[\mathfrak{H}(T)=\max\{\eta(v)\mid v\in V(T)\}.\]    
\end{definition}	

 \begin{theorem}\label{thm:drevoZgMeja}
		If $T$ is a tree such that $\svSpan{T} \geq 2$, then $\svSpan{T} \leq \mathfrak{H}(T) $.
	\end{theorem}
	\begin{proof}
        Note, since $\svSpan{T} \geq 2$, it follows that $T$ is not a path \cite{BaTa23} and it therefore contains a vertex of degree at least 3.
        
		Suppose on the contrary that there exists a path $P$ on $n$ vertices and surjective weak homomorphisms $A:P\longrightarrow T$ and $B:P\longrightarrow T$ such that $d(A(t),B(t))\geq \mathfrak{H}(T)+1$ for all $t\in V(P)$. 

        Define $\widehat{t}$ to be  
		\[\widehat{t}=\min\{t \in V(P) \mid B(t) \mbox{ is on the path between } A(1) \mbox{ and } A(t)\] \[\mbox{ or } A(t) \mbox{ is on the path between } B(1) \mbox{ and } B(t)\}.\]
		
		Without loss of generality assume $B(\widehat{t})$ is on the path between $A(1)$  and  $A(\widehat{t})$. Due to surjectivity such a $\widehat{t}$ exists because $B(t)=A(1)$ for some $t\in V(P)$. By Lemma \ref{switch} we have that B switches with A at $\widehat{v}=B(\widehat{t})$.

        Following the notation in the proof of Lemma \ref{switch}, let $C_{\alpha}(\widehat{v}), C_{\beta}(\widehat{v}), C_{\gamma}(\widehat{v})$ be the distinct components of $T-\{\widehat{v}\}$ such that $A(1)\in C_{\alpha}(\widehat{v})$, $B(\widehat{t}-1)\in C_{\beta}(\widehat{v})$, $A(\widehat{t})\in C_{\gamma}(\widehat{v})$. Also define \[t_0 = \max\{t \in V(P) \mid 1 \leq t < \widehat{t} \text{ and } A(t) \in C_\alpha(\widehat{v})\}.\] As stated as in the proof of Lemma \ref{switch}, $t_0$ exists since $A(1)\in C_\alpha(\widehat{v})$ and $B(t_0)\in C_\beta(\widehat{v})$. We now have, $A(t_0)\in C_{\alpha}(\widehat{v})$, $B(t_0)\in C_{\beta}(\widehat{v})$, $A(\widehat{t})\in C_{\gamma}(\widehat{v})$.
        
		Since $B$ switches with $A$ at $\widehat{v}$ it follows from Lemma \ref{leg} that 
        \begin{equation}\label{neenakost:BswitchAL21}
            \min\{d(A(t),B(t))\mid t\in V(P)\} \leq \min\{\reach(C_{\beta}(\widehat{v})),\reach(C_{\gamma}(\widehat{v}))\}.
        \end{equation}

        Using notations from Definition \ref{def:triod-max-eta} let $u \in V(T)$ be such that $\reach(C_3(u))=\mathfrak{H}(T)$. Hence for any $v\in V(G), \reach(C_3(u))\geq \reach(C_3(v))$. Using the starting hypothesis we have that for all $v \in V(T)$
        \begin{align*}
            \min\{d(A(t), B(t))\mid t\in V(P)\} & \geq \reach(C_3(u)) + 1 \\
                                              & \geq \reach(C_3(v)) + 1 \\
                                              & \geq \reach(C_k(v)) + 1, \text{ for all }k\geq 3.
        \end{align*}
        Hence using \eqref{neenakost:BswitchAL21}, for every  $C_{\delta}(\widehat{v})\not\in\{C_{\beta}(\widehat{v}), C_{\gamma}(\widehat{v})\}$ we have
        \begin{equation}\label{pogoj:zaProtislovjeT32}
            \min\{d(A(t),B(t))\mid t\in V(P)\} \geq \reach(C_{\delta}(\widehat{v}))+1. 
        \end{equation}

		Let $\widehat{B}$ be a vertex of $C_{\beta}(\widehat{v})$ such that $d(\widehat{B}, \widehat{v})=\reach(C_{\beta}(\widehat{v}))$. Let 
		\[\widetilde{t}=\min\{t\in V(P)\mid  A(t) \mbox{ is on the path between } \widehat{B} \mbox{ and } B(t)\}.\]
		
		
		Let $\widetilde{v}=A(\widetilde{t})$. We now consider all cases of possible locations for $\widetilde{v}$. 

        \begin{case}
		Assume $\widetilde{v}=\widehat{v}$. Since $\widehat{B}\in C_{\beta}(\widehat{v})$ and $\widetilde{v}$ is on the path between $\widehat{B}$ and $B(\widetilde{t})$, it follows that $B(\widetilde{t})\not\in  C_{\beta}(\widehat{v})$. If $B(\widetilde{t})\in C_{\delta}(\widehat{v})$, for any $C_{\delta}(\widehat{v})\not\in\{C_{\beta}(\widehat{v}), C_{\gamma}(\widehat{v})\}$, then $d(A(\widetilde{t}),B(\widetilde{t}))\leq \reach(C_{\delta}(\widehat{v}))$,  which is impossible due to condition \eqref{pogoj:zaProtislovjeT32}. Hence, $B(\widetilde{t})\in C_{\gamma}(\widehat{v})$. Let $t'=\max\{t \in V(P) \mid t<\widetilde{t} \text{ and } B(t)=\widehat{v}\}$. Thus,  $B(t)\in C_{\gamma}(\widehat{v})$ for all $t'<t\leq \widetilde{t}$. 
  
        Depending on the location of $A(t')$ we consider all the possibilities. 

        If $A(t')\in C_{\delta}(\widehat{v})$, for any $C_{\delta}(\widehat{v})\not\in\{C_{\beta}(\widehat{v}), C_{\gamma}(\widehat{v})\}$, then $d(A(t'),B(t'))\leq \reach(C_{\delta}(\widehat{v}))$,  which is impossible due to \eqref{pogoj:zaProtislovjeT32}.

        If $A(t')\in C_{\delta}(\widehat{v})$, for any $C_{\delta}(\widehat{v})\in\{C_{\beta}(\widehat{v}), C_{\gamma}(\widehat{v})\}$, then there exists $t''$ such that $t'<t''<\widetilde{t}$ and $A(t'')\in \bor(C_{\delta}(\widehat{v}))$. Since $B(t'')\in C_{\gamma}(\widehat{v})$, $A(t'')$ is on the path from $\widehat{B}$ to $B(t'')$ which contradicts the definition of $\widetilde{t}$. See Figure \ref{fig:case1} for a sketch with all the important components and vertices shown in the case when $A(t')\in C_{\beta}(\widehat{v})$.

        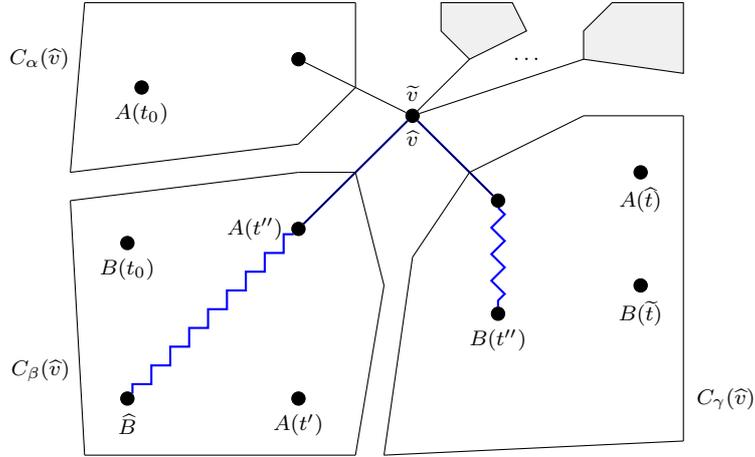
\begin{figure}[!ht]
            \centering
    \begin{tikzpicture}[scale=0.75]
            \footnotesize
\tikzstyle{rn}=[circle, fill=black,draw, inner sep=0pt, minimum size=5pt]
\tikzstyle{bobT}=[fill={rgb,255: red,255; green,128; blue,0}, draw=black, inner sep=0pt, shape=circle, minimum size=5pt]

\tikzstyle{dashed grey}=[-, draw={rgb,255: red,128; green,128; blue,128}, dashed]

	\begin{pgfonlayer}{nodelayer}
		\node (1) at (1, 1) {};
		\node (2) at (2, 1) {$\ldots$};
		\node (3) at (3, 1) {};
		\node (4) at (0.5, 1.5) {};
		\node (5) at (0.5, 2) {};
		\node (6) at (1.75, 2) {};
		\node (7) at (2, 1.5) {};
		\node (8) at (3, 1.5) {};
		\node (9) at (3.5, 2) {};
		\node (10) at (4.75, 2) {};
		\node (11) at (4.75, 0.75) {};
		
		\node (13) at (-1, 0.5) {};
		\node (14) at (-1, 2) {};
		\node (15) at (-5.75, 2) {};
		\node (16) at (-6, -1) {};
		\node (17) at (-2, -0.5) {};
		
		\node (19) at (-7, 1) {};
		
		\node (23) at (-1, -1) {};
		\node (24) at (-6, -1.5) {};
		\node (25) at (-2, -1) {};
		\node (26) at (-5.75, -6) {};
		\node (27) at (-1, -6) {};
		\node (28) at (-0.5, -3) {};
		
		\node (32) at (0, -2.5) {};
		\node (33) at (-0.5, -6) {};
		\node (34) at (4.75, -5.75) {};
		\node (35) at (4.75, 0) {};
		\node (36) at (3, 0) {};
		\node (37) at (1, -1) {};

        \node [style=rn] (18) at (-4.75, 0.5) [label=below:$A(t_0)$]{};
        \node (20) at (-6.54, 1) {$C_\alpha(\widehat{v})$};
        \node [style=rn] (29) at (-5, -5) [label=below:$\widehat{B}$]{};
		\node [style=rn] (30) at (-5, -2.25) [label=below:$B(t_0)$]{};
		\node [style=rn] (38) at (1.5, -1.5) {};
  		\node [style=rn] (btc) at (1.5, -3.5) [label=below:$B(t'')$]{};
        \node [style=rn] (21) at (-2, -2) [label=left:$A(t'')$]{};
        \node [style=rn] (atc) at (-2, -5) [label=below:$A(t')$]{};
        \node [style=rn] (12) at (-2, 1) {};
        \node [style=rn] (0) at (0, 0) [label=below:{$\widehat{v}$}, label=above:{$\widetilde{v}$}]{};
		\node [style=rn] (39) at (4, -1) [label=below:$A(\widehat{t})$]{};
		\node (40) at (-6.52, -4.5) {$C_\beta(\widehat{v})$};
		\node (41) at (5.5, -5) {$C_\gamma(\widehat{v})$};
          \node [style=rn] (btv) at (4, -3) [label=below:$B(\widetilde{t})$]{};
	\end{pgfonlayer}
	\begin{pgfonlayer}{edgelayer}
        \draw[decorate, thick, decoration=zigzag, color=blue] (29)--(21);
         \draw[decorate, thick, color=blue] (21) -- (0) -- (38);
         \draw[decorate, thick, decoration=zigzag, color=blue] (38)--(btc);
 
		\draw (0) to (1.center);
		\draw (0) to (3.center);
  
		\filldraw[fill={rgb,255: red,240; green,240; blue,240}] (1.center) -- (4.center) -- (5.center)-- (6.center)-- (7.center) --cycle;

		\filldraw[fill={rgb,255: red,240; green,240; blue,240}] (3.center) -- (8.center) -- (9.center) -- (10.center) -- (11.center) -- cycle;

		\draw (14.center) to (13.center);
		\draw (13.center) to (17.center);
		\draw (17.center) to (16.center);
		\draw (16.center) to (15.center);
		\draw (15.center) to (14.center);
		\draw (12) to (0);
		\draw (0) to (21);
		\draw (23.center) to (25.center);
		\draw (25.center) to (24.center);
		\draw (24.center) to (26.center);
		\draw (26.center) to (27.center);
		\draw (27.center) to (28.center);
		\draw (28.center) to (23.center);
		\draw (37.center) to (32.center);
		\draw (32.center) to (33.center);
		\draw (33.center) to (34.center);
		\draw (34.center) to (35.center);
		\draw (35.center) to (36.center);
		\draw (36.center) to (37.center);
		\draw (0) to (38);
	\end{pgfonlayer}
\end{tikzpicture}
            \caption{Case 1 when $A(t')\in C_{\beta}(\widehat{v})$.}
            \label{fig:case1}
        \end{figure}
        
		\end{case}

        \begin{case}
        Assume $\widetilde{v}\in C_{\delta}(\widehat{v})$, for any $C_{\delta}(\widehat{v})\not\in\{C_{\beta}(\widehat{v}), C_{\gamma}(\widehat{v})\}$. Then $B(\widetilde{t})\in C_{\delta}(\widehat{v})$ and hence $\widetilde{v}$ in on a path from $B(\widetilde{t})$ to $\widehat{v}$. Hence, $d(A(\widetilde{t}),B(\widetilde{t}))\leq \reach(C_{\delta}(\widehat{v}))$,  which is impossible due to \eqref{pogoj:zaProtislovjeT32}. See Figure \ref{fig:case2-alpha} for an example where $\widetilde{v}\in C_{\alpha}(\widehat{v})$.

        \begin{figure}[!ht]
            \centering
    \begin{tikzpicture}[scale=0.75]
            \footnotesize
\tikzstyle{rn}=[circle, fill=black,draw, inner sep=0pt, minimum size=5pt]
\tikzstyle{bobT}=[fill={rgb,255: red,255; green,128; blue,0}, draw=black, inner sep=0pt, shape=circle, minimum size=5pt]

\tikzstyle{dashed grey}=[-, draw={rgb,255: red,128; green,128; blue,128}, dashed]

	\begin{pgfonlayer}{nodelayer}
		\node (1) at (1, 1) {};
		\node (2) at (2, 1) {$\ldots$};
		\node (3) at (3, 1) {};
		\node (4) at (0.5, 1.5) {};
		\node (5) at (0.5, 2) {};
		\node (6) at (1.75, 2) {};
		\node (7) at (2, 1.5) {};
		\node (8) at (3, 1.5) {};
		\node (9) at (3.5, 2) {};
		\node (10) at (4.75, 2) {};
		\node (11) at (4.75, 0.75) {};
		
		\node (13) at (-1, 0.5) {};
		\node (14) at (-1, 2) {};
		\node (15) at (-5.75, 2) {};
		\node (16) at (-6, -1) {};
		\node (17) at (-2, -0.5) {};
		
		\node (19) at (-7, 1) {};
		
		\node (23) at (-1, -1) {};
		\node (24) at (-6, -1.5) {};
		\node (25) at (-2, -1) {};
		\node (26) at (-5.75, -6) {};
		\node (27) at (-1, -6) {};
		\node (28) at (-0.5, -3) {};
		
		\node (32) at (0, -2.5) {};
		\node (33) at (-0.5, -6) {};
		\node (34) at (4.75, -5.75) {};
		\node (35) at (4.75, 0) {};
		\node (36) at (3, 0) {};
		\node (37) at (1, -1) {};

        \node [style=rn] (18) at (-4.75, 0.5) [label=below:$A(t_0)$]{};
        \node (20) at (-6.5, 1) {$C_\alpha(\widehat{v})$};
        \node [style=rn] (29) at (-5.25, -5.2) [label=below:$\widehat{B}$]{};
		\node [style=rn] (30) at (-5, -2.25) [label=below:$B(t_0)$]{};
		\node [style=rn] (38) at (1.5, -1.5) {};
        \node [style=rn] (21) at (-2, -2) {};
        \node [style=rn] (12) at (-2, 1) {};
        \node [style=rn] (0) at (0, 0) [label=below:{$\widehat{v}$}]{};
		\node [style=rn] (39) at (4, -1) [label=below:$A(\widehat{t})$]{};
		\node (40) at (-6.5, -4) {$C_\beta(\widehat{v})$};
		\node (41) at (5.5, -4) {$C_\gamma(\widehat{v})$};
        \node [style=rn] (vv) at (-3.5, 1) [label=below:$\widetilde{v}$, label=above:$A(\widetilde{t})$]{};
        \node [style=rn] (btv) at (-5, 1.5) [label=below:$B(\widetilde{t})$]{};
	\end{pgfonlayer}
	\begin{pgfonlayer}{edgelayer}
        \draw[decorate, thick, decoration=zigzag, color=blue] (29)--(21);
        \draw[decorate, thick, color=blue] (21) -- (0) -- (12);
        \draw[decorate, thick, decoration=zigzag, color=blue] (12)--(vv)--(btv);
        
		\draw (0) to (1.center);
		\draw (0) to (3.center);
  
		\filldraw[fill={rgb,255: red,240; green,240; blue,240}] (1.center) -- (4.center) -- (5.center)-- (6.center)-- (7.center) --cycle;

		\filldraw[fill={rgb,255: red,240; green,240; blue,240}] (3.center) -- (8.center) -- (9.center) -- (10.center) -- (11.center) -- cycle;

		\draw (14.center) to (13.center);
		\draw (13.center) to (17.center);
		\draw (17.center) to (16.center);
		\draw (16.center) to (15.center);
		\draw (15.center) to (14.center);
		\draw (23.center) to (25.center);
		\draw (25.center) to (24.center);
		\draw (24.center) to (26.center);
		\draw (26.center) to (27.center);
		\draw (27.center) to (28.center);
		\draw (28.center) to (23.center);
		\draw (37.center) to (32.center);
		\draw (32.center) to (33.center);
		\draw (33.center) to (34.center);
		\draw (34.center) to (35.center);
		\draw (35.center) to (36.center);
		\draw (36.center) to (37.center);
		\draw (0) to (38);
	\end{pgfonlayer}
\end{tikzpicture}

            \caption{Case 2 when $\widetilde{v}\in C_{\alpha}(\widehat{v})$.}
            \label{fig:case2-alpha}
        \end{figure}
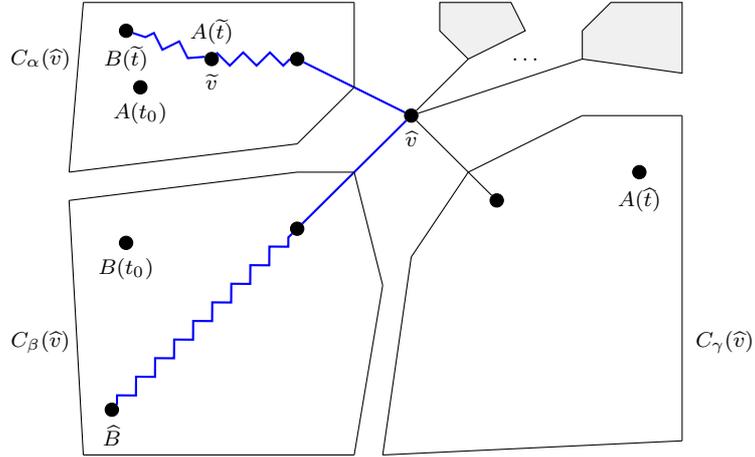
        \end{case}

        \begin{case}
        Assume $\widetilde{v}\in C_{\gamma}(\widehat{v})$. Let $C_{\widehat{\alpha}}(\widetilde{v})$ be the component of $T-\{\widetilde{v}\}$ that contains $\widehat{v}$. Then both $C_{\alpha}(\widehat{v})$ and $C_{\beta}(\widehat{v})$ are contained in $C_{\widehat{\alpha}}(\widetilde{v})$. From \eqref{neenakost:BswitchAL21} and since $d(\widehat{v},\widetilde{v})\geq 1$, it follows that  \[\reach(C_{\widehat{\alpha}}(\widetilde{v})) > \reach(C_{\beta}(\widehat{v}))\geq \min\{d(A(t),B(t))\mid t\in P\}.\]
		Let $C_{\widehat{\beta}}(\widetilde{v})$ be the component that contains $A(\widetilde{t}-1)$ and  $C_{\widehat{\gamma}}(\widetilde{v})$ be the component that contains  $B(\widetilde{t}-1)$ (and therefore $B(\widetilde{t})$). The components $C_{\widehat{\alpha}}(\widetilde{v}), C_{\widehat{\beta}}(\widetilde{v})$ and $C_{\widehat{\gamma}}(\widetilde{v})$ must all be different or the definition of $\widetilde{t}$ is violated. See Figure \ref{fig:case3} for a sketch of the important components and vertices for this case.

        \begin{figure}[!ht]
            \centering
    \begin{tikzpicture}[scale=0.75]
            \footnotesize
\tikzstyle{rn}=[circle, fill=black,draw, inner sep=0pt, minimum size=5pt]
\tikzstyle{bobT}=[fill={rgb,255: red,255; green,128; blue,0}, draw=black, inner sep=0pt, shape=circle, minimum size=5pt]

\tikzstyle{dashed grey}=[-, draw={rgb,255: red,128; green,128; blue,128}, dashed]

	\begin{pgfonlayer}{nodelayer}
		\node (1) at (1, 1) {};
		\node (2) at (2, 1) {$\ldots$};
		\node (3) at (3, 1) {};
		\node (4) at (0.5, 1.5) {};
		\node (5) at (0.5, 2) {};
		\node (6) at (1.75, 2) {};
		\node (7) at (2, 1.5) {};
		\node (8) at (3, 1.5) {};
		\node (9) at (3.5, 2) {};
		\node (10) at (4.75, 2) {};
		\node (11) at (4.75, 0.75) {};
		
		\node (13) at (-1, 0.5) {};
		\node (14) at (-1, 2) {};
		\node (15) at (-5.75, 2) {};
		\node (16) at (-6, -1) {};
		\node (17) at (-2, -0.5) {};
		
		\node (19) at (-7, 1) {};
		
		\node (23) at (-1, -1) {};
		\node (24) at (-6, -1.5) {};
		\node (25) at (-2, -1) {};
		\node (26) at (-5.75, -4) {};
		\node (27) at (-3.5, -3.5) {};
		\node (28) at (-0.5, -1.5) {};
		
		\node (32) at (-5.5, -5.5) {};
		\node (33) at (-3.5, -6) {};
		\node (34) at (4.75, -5.75) {};
		\node (35) at (4.75, 0) {};
		\node (36) at (3, 0) {};
		\node (37) at (1, -1) {};

        \node [style=rn] (18) at (-4.75, 0.5) [label=below:$A(t_0)$]{};
        \node (20) at (-6.5, 1) {$C_\alpha(\widehat{v})$};
        \node [style=rn] (29) at (-5.25, -3) [label=below:$\widehat{B}$]{};
		\node [style=rn] (30) at (-5, -1.75) [label=below:$B(t_0)$]{};
		\node [style=rn] (38) at (1.75, -1.75) [label=right:$B(t_1)$]{};
        \node [style=rn] (21) at (-2, -2) {};
        \node [style=rn] (12) at (-2, 1) {};
        \node [style=rn] (0) at (0, 0) [label=below:{$\widehat{v}$}]{};
		\node [style=rn] (39) at (4, -1) [label=below:$A(\widehat{t})$]{};
		\node (40) at (-6.5, -3) {$C_\beta(\widehat{v})$};
		\node (41) at (5.5, -4) {$C_\gamma(\widehat{v})$};
        \node [style=rn] (vv) at (2.5, -2.5) [label=below:$\widetilde{v}$]{};
        \node [style=rn] (btv) at (0, -4.5) [label=below:$B(\widetilde{t})$]{};
        \node [style=rn] (btv1) at (0.5, -2.75) [label=below:$B(\widetilde{t}-1)$]{};
        \node [style=rn] (atv1) at (3.5, -3.5) [label=below:$A(\widetilde{t}-1)$]{};
        \node [style=rn] (at1) at (4, -5) [label=left:$A(t_1)$]{};
        \node [color = blue](obs) at (2, -5.5) {$C_{\widehat{\beta}}(\widetilde{v})$};
        \node [color = blue](ogs) at (-2.5, -5.5) {$C_{\widehat{\gamma}}(\widetilde{v})$};

        \node [color = blue](oalfahat) at (-7.75, -5) {$C_{\widehat{\alpha}}(\widetilde{v})$};
	\end{pgfonlayer}
	\begin{pgfonlayer}{edgelayer}
        \draw (vv)--(2.75, -2.1);
        \draw (vv)--(2.9, -2.3);
        \draw (vv)--(3, -2.5);
        \draw[pattern=dots, pattern color=green, dotted] (0, -0.75) -- (-0.2, -1.5) -- (-7, -6.25) -- (-7.5, 3) -- (5.25, 3) -- (5.25, 0.5) -- (0.75, 0) -- (2.5, -2) -- (2, -2.5) -- cycle;

        \draw[pattern=horizontal lines light blue, pattern color=gray, dotted] (3, -3) -- (4.75, -3) -- (4.74, -5.74) -- (1, -5.85) -- cycle;

        \draw[pattern=horizontal lines, pattern color=yellow, dotted] (2, -2.75) -- (0.75, -5.82) -- (-3.5, -6) -- (-3, -4) -- (0, -2) -- cycle;
        
        \draw[decorate, thick, color=blue] (vv)--(38);
        \draw[decorate, thick, decoration=zigzag, color=blue] (29)--(21);
        \draw[decorate, thick, decoration=zigzag, color=blue] (vv)--(btv);
        \draw[decorate, thick, color=blue] (21) -- (0);
        \draw[decorate, thick, decoration=zigzag, color=blue] (0) to (38);
        
		\draw (12)--(0) to (1.center);
		\draw (0) to (3.center);
        \draw (vv)--(atv1);
  
		\draw (1.center) -- (4.center) -- (5.center)-- (6.center)-- (7.center) --cycle;

		\draw (3.center) -- (8.center) -- (9.center) -- (10.center) -- (11.center) -- cycle;

		\draw (14.center) to (13.center);
		\draw (13.center) to (17.center);
		\draw (17.center) to (16.center);
		\draw (16.center) to (15.center);
		\draw (15.center) to (14.center);

		\draw (23.center) to (25.center);
		\draw (25.center) to (24.center);
		\draw (24.center) to (26.center);
		\draw (26.center) to (27.center);
		\draw (27.center) to (28.center);
		\draw (28.center) to (23.center);

		\draw (37.center) to (32.center);
		\draw (32.center) to (33.center);
		\draw (33.center) to (34.center);
		\draw (34.center) to (35.center);
		\draw (35.center) to (36.center);
		\draw (36.center) to (37.center);
	
	\end{pgfonlayer}
\end{tikzpicture}

            \caption{A sketch of the graph structure for the Case 3.}
            \label{fig:case3}
        \end{figure}
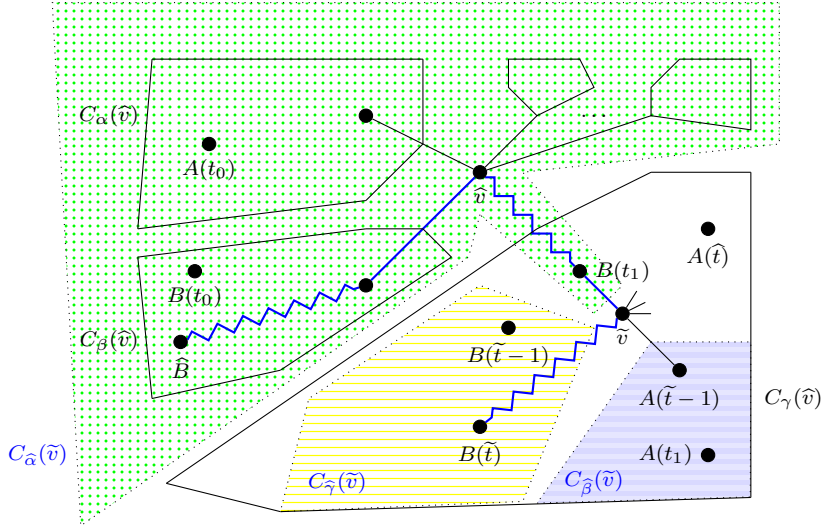
		
		Notice that $A(\widetilde{t}-1)\in \bor(C_{\widehat{\beta}}(\widetilde{v}))$. Since $B(t_0)\in C_{\beta}(\widehat{v})\subset C_{\widehat{\alpha}}(\widetilde{v})$, there exists $t_1=\max\{t \in V(P) \mid  t<\widetilde{t}  \text{ and } B(t) \in C_{\widehat{\alpha}}(\widetilde{v})\}$. Then $B(t_1)\in \bor(C_{\widehat{\alpha}}(\widetilde{v}))$ and $B(t_1+1)=\widetilde{v}$.

        \medskip
		
		{\bf Claim:} $A(t_1)\in C_{\widehat{\beta}}(\widetilde{v})$.

        \medskip
  
		If  $A(t_1)\in C_{\widehat{\alpha}}(\widetilde{v})$, then there exists $t_1<t'<\widetilde{t}$ such that $A(t')\in \bor(C_{\widehat{\alpha}}(\widetilde{v}))$.  Since $t_1<t'<\widetilde{t}$, by definition of $t_1$ it follows that $B(t')\not \in C_{\widehat{\alpha}}(\widetilde{v})$. This, together with the fact that $\widehat{B}\in C_{\widehat{\alpha}}(\widetilde{v})$ and $A(t')\in\bor(C_{\widehat{\alpha}}(\widetilde{v}))$, implies that $A(t')$ is on the path from $\widehat{B}$ to $B(t')$ which is a contradiction with the definition of $\widetilde{t}$.
		
		If $A(t_1)\in C_{\widehat{\delta}}(\widetilde{v})$, for any $C_{\widehat{\delta}}(\widehat{v})\not\in\{C_{\widehat{\alpha}}(\widetilde{v}), C_{\widehat{\beta}}(\widetilde{v})\}$, then because $A(\widetilde{t}-1)\in C_{\widehat{\beta}}(\widetilde{v})$ there exists $t_1<t''<\widetilde{t}$ such that $A(t'')=\widetilde{v}$. Again since $t_1<t''<\widetilde{t}$, by definition of $t_1$ it follows that $B(t'')\not \in C_{\widehat{\alpha}}(\widetilde{v})$, hence $A(t'')$ is on the path from $\widehat{B}$ to $B(t'')$ which contradicts the definition of $\widetilde{t}$. Hence, the assertion of the claim is proved. 

        \medskip
        
		Using $A(t_1)\in C_{\widehat{\beta}}(\widetilde{v})$, $B(t_1)\in C_{\widehat{\alpha}}(\widetilde{v})$ and $B(\widetilde{t})\in C_{\widehat{\gamma}}(\widetilde{v})$, it follows from definition of $\widetilde{t}$ that the conditions for Lemma \ref{switch} are satisfied, therefore $A$ switches with $B$ at $\widetilde{v}$. By Lemma \ref{leg}, $\min\{d(A(t),B(t))\mid t\in V(P)\} \leq \min\{\reach(C_{\widehat{\beta}}(\widetilde{v})),\reach(C_{\widehat{\gamma}}(\widetilde{v}))\}.$
			Also, since $C_{\beta}(\widehat{v})\subset C_{\widehat{\alpha}}(\widetilde{v})$, using inequality \eqref{neenakost:BswitchAL21} we find that  $\min\{d(A(t),B(t))\mid t\in P\} \leq \reach(C_{\widehat{\alpha}}(\widetilde{v})).$ Hence,
		\[\min\{d(A(t),B(t))\mid t\in V(P)\} \leq \eta(\widetilde{v})\leq \mathfrak{H}(T),\]
		which is in contradiction with the starting hypothesis.

        \end{case}
		\begin{case}
		    Assume $\widetilde{v}\in C_{\beta}(\widehat{v})$. We have to consider two subcases.
            
            \begin{subcase}
            \begin{figure}[!ht]
            \centering
    \begin{tikzpicture}[scale=0.75]
            \footnotesize
\tikzstyle{rn}=[circle, fill=black,draw, inner sep=0pt, minimum size=5pt]
\tikzstyle{bobT}=[fill={rgb,255: red,255; green,128; blue,0}, draw=black, inner sep=0pt, shape=circle, minimum size=5pt]

\tikzstyle{dashed grey}=[-, draw={rgb,255: red,128; green,128; blue,128}, dashed]

	\begin{pgfonlayer}{nodelayer}
		\node (1) at (1, 1) {};
		\node (2) at (2, 1) {$\ldots$};
		\node (3) at (3, 1) {};
		\node (4) at (0.5, 1.5) {};
		\node (5) at (0.5, 2) {};
		\node (6) at (1.75, 2) {};
		\node (7) at (2, 1.5) {};
		\node (8) at (3, 1.5) {};
		\node (9) at (3.5, 2) {};
		\node (10) at (4.75, 2) {};
		\node (11) at (4.75, 0.75) {};
		
		\node (13) at (-1, 0.5) {};
		\node (14) at (-1, 2) {};
		\node (15) at (-5.75, 2) {};
		\node (16) at (-6, -1) {};
		\node (17) at (-2, -0.5) {};
		
		\node (19) at (-7, 1) {};
		
		\node (23) at (-1, -1) {};
		\node (24) at (-6, -1.5) {};
		\node (25) at (-2, -1) {};
		\node (26) at (-5.75, -6) {};
		\node (27) at (4.75, -5.75) {};
		\node (28) at (4.75, -3.5) {};
		
		\node (32) at (1, -0.5) {};
		\node (33) at (2, -1.5) {};
		\node (34) at (4.75, -2.75) {};
		\node (35) at (4.75, 0) {};
		\node (36) at (3, 0) {};
		\node (37) at (1, -0.25) {};

        \node [style=rn] (18) at (-4.75, 0.5) [label=below:$A(t_0)$]{};
        \node (20) at (-6.5, 1) {$C_\alpha(\widehat{v})$};
        \node [style=rn] (29) at (-5.25, -5) [label=below:$\widehat{B}$]{};
		\node [style=rn] (30) at (-5, -2.5) [label=below:$B(t_0)$]{};
		\node [style=rn] (38) at (-2, -2.5) [label=right:$A(\widetilde{t}-1)$]{};
        \node (21) at (-1.5, -1.5) {};
        \node (12) at (-1, 0.5) {};
        \node [style=rn] (0) at (0, 0) [label=below:{$\widehat{v}$}]{};
		\node [style=rn] (39) at (4, -1) [label=below:$A(\widehat{t})$]{};
		\node (40) at (-6.5, -3) {$C_\beta(\widehat{v})$};
		\node (41) at (5.5, -1.5) {$C_\gamma(\widehat{v})$};
        \node [style=rn] (vv) at (-1, -3.5) [label=below:$\widetilde{v}$]{};
        \node [style=rn] (btv) at (2, -4.5) [label=below:$B(\widetilde{t})$]{};
        \node [style=rn] (btv1) at (1.75, -3.75) [label=right:$B(\widetilde{t}-1)$]{};
         \node [color = blue](obs) at (3.75, -5.25) {$C_{\widehat{\beta}}(\widetilde{v})$};
        \node [color = blue](ogs) at (-3.5, -5.5) {$C_{\widehat{\gamma}}(\widetilde{v})$};

        \node [color = blue](oalfahat) at (-7.75, -1) {$C_{\widehat{\alpha}}(\widetilde{v})$};
	\end{pgfonlayer}
	\begin{pgfonlayer}{edgelayer}
        \draw (vv)--(-0.75, -3.1); 
        \draw (vv)--(-0.6, -3.3);
        \draw (vv)--(-1, -3);
        \draw[pattern=dots, pattern color=green, dotted] (-0.75, -0.5) -- (-7, -1.5) -- (-7.5, 3) -- (5.5, 3) -- (6.25, -1.5) -- (5, -3) -- (0, -1) --(-1.75, -2.75)-- (-2.5, -2.75) -- cycle;

        \draw[pattern=horizontal lines light blue, pattern color=gray, dotted] (-0.5, -3.75) -- (1, -3) -- (3, -2.75) -- (4.74, -5.5) -- (1, -5.75) -- cycle;

        \draw[pattern=horizontal lines, pattern color=yellow, dotted] (-2, -5) -- (-3, -5.82) -- (-5.5, -5.82) -- (-5.5, -4) -- (-1.5, -3.5) -- cycle;
        
        \draw[decorate, thick, color=blue] (vv)--(38);
        \draw[decorate, thick, decoration=zigzag, color=blue] (29)--(vv);
        \draw[decorate, thick, decoration=zigzag, color=blue] (vv)--(btv);
        \draw[decorate, thick, decoration=zigzag, color=blue] (0) to (38);
        
		\draw (12.center)--(0) to (1.center);
        \draw (0)--(37.center);
		\draw (0) to (3.center);
  
		\draw (1.center) -- (4.center) -- (5.center)-- (6.center)-- (7.center) --cycle;

		\draw (3.center) -- (8.center) -- (9.center) -- (10.center) -- (11.center) -- cycle;

		\draw (14.center) to (13.center);
		\draw (13.center) to (17.center);
		\draw (17.center) to (16.center);
		\draw (16.center) to (15.center);
		\draw (15.center) to (14.center);

		\draw (23.center) to (25.center);
		\draw (25.center) to (24.center);
		\draw (24.center) to (26.center);
		\draw (26.center) to (27.center);
		\draw (27.center) to (28.center);
		\draw (28.center) to (23.center);

		\draw (37.center) to (32.center);
		\draw (32.center) to (33.center);
		\draw (33.center) to (34.center);
		\draw (34.center) to (35.center);
		\draw (35.center) to (36.center);
		\draw (36.center) to (37.center);
	
	\end{pgfonlayer}
\end{tikzpicture}

            \caption{Situation when $\widetilde{v}\in C_{\beta}(\widehat{v})$ and $A(\widetilde{t}-1)$ is in the path from $\widehat{B}$ to $\widehat{v}$ (Subcase 4.1.).}
            \label{fig:case4-1}
        \end{figure}
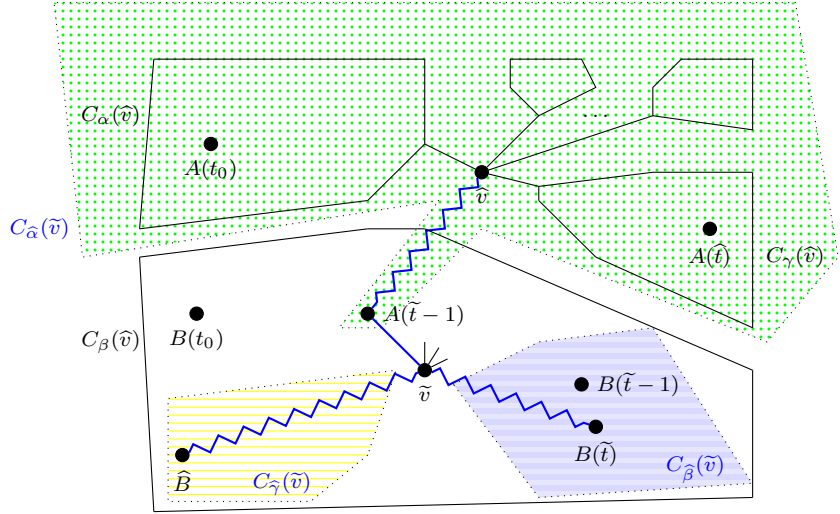
		
             $A(\widetilde{t}-1)$ is in the path from $\widehat{B}$ to $\widehat{v}$. By the choice of $\widetilde{t}$ it follows that $\widetilde{v}$ is also in the path from $\widehat{B}$ to $\widehat{v}$. Then $B(\widetilde{t}-1)$ and hence $B(\widetilde{t})$ are in a different component of $T-\{\widetilde{v}\}$ than both $\widehat{v}$ and $\widehat{B}$. Let $C_{\widehat{\alpha}}(\widetilde{v})$ be the component that contains $\widehat{v}$, $C_{\widehat{\beta}}(\widetilde{v})$ be the component that contains $B(\widetilde{t})$ and $C_{\widehat{\gamma}}(\widetilde{v})$ be the component that contains $\widehat{B}$, see Figure \ref{fig:case4-1}. First, $d(\widetilde{v},B(\widetilde{t})) \leq \reach(C_{\widehat{\beta}}(\widetilde{v}))$. Also, $d(\widehat{B},\widehat{v})\geq d(v,\widehat{v})$  for every $v\in C_{\widehat{\beta}}(\widetilde{v})\subset C_{\beta}(\widehat{v})$ and $d(\widehat{B},\widehat{v})=d(\widehat{B},\widetilde{v})+d(\widetilde{v},\widehat{v})$. Furthermore,  $d(v,\widehat{v})=d(v,\widetilde{v})+d(\widetilde{v},\widehat{v})$ for every $v\in C_{\widehat{\beta}}(\widetilde{v})$, hence $\reach(C_{\widehat{\beta}}(\widetilde{v}))\leq \reach(C_{\widehat{\gamma}}(\widetilde{v}))$. Finally, $C_{\gamma}(\widehat{v})\subset C_{\widehat{\alpha}}(\widetilde{v})$, therefore $\reach(C_{\gamma}(\widehat{v}))\leq \reach(C_{\widehat{\alpha}}(\widetilde{v}))$.

             Using, \eqref{neenakost:BswitchAL21} and $d(A(\widetilde{t}), B(\widetilde{t}))\leq \reach(C_{\widehat{\beta}}(\widetilde{v})) \leq \reach(C_{\widehat{\gamma}}(\widetilde{v}))$, 
             we conclude that 
     \[\min\{d(A(t),B(t))\mid t\in P\} \leq \min\{\reach(C_{\widehat{\alpha}}(\widetilde{v})),\reach(C_{\widehat{\beta}}(\widetilde{v})),\reach(C_{\widehat{\gamma}}(\widetilde{v}))\}\leq \eta(\widetilde{v})\leq \mathfrak{H}(T),\]
    which is a contradiction with the starting hypothesis.   
            \end{subcase}

            \begin{subcase}
             $A(\widetilde{t}-1)$ is not in the path from $\widehat{B}$ to $\widehat{v}$. Then $A(\widetilde{t}-1)$ is in a different component of $T-\{\widetilde{v}\}$ than both $\widehat{v}$ and $\widehat{B}$.  With respect to $T-\{\widetilde{v}\}$, let $C_{\widehat{\alpha}}(\widetilde{v})$ be the component that contains $\widehat{v}$,  $C_{\widehat{\beta}}(\widetilde{v})$ be the component that contains $\widehat{B}$,  $C_{\widehat{\gamma}}(\widetilde{v})$ be the component that contains $A(\widetilde{t}-1)$ and $C_{\widehat{\delta}}(\widetilde{v})$ be the component that contains $B(\widetilde{t}-1)$ and hence $B(\widetilde{t})$. Note that $C_{\widehat{\alpha}}(\widetilde{v})$ may equal $C_{\widehat{\delta}}(\widetilde{v})$ (depicted in Figure \ref{fig:case4-2b}), also $C_{\widehat{\alpha}}(\widetilde{v})$ may equal $C_{\widehat{\beta}}(\widetilde{v})$ (depicted in Figure \ref{fig:case4-2a1}).  These two possibilities are mutually exclusive.

      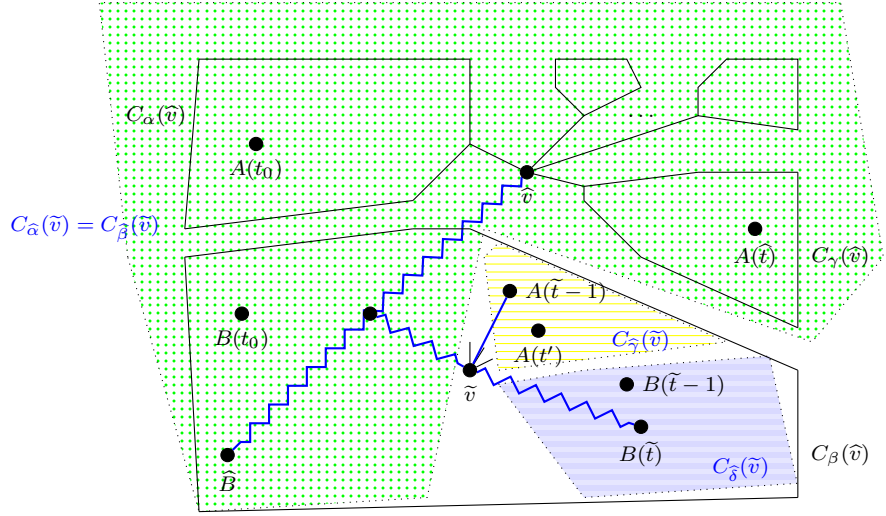
\begin{figure}[!ht]
            \centering
    \begin{tikzpicture}[scale=0.75]
            \footnotesize
\tikzstyle{rn}=[circle, fill=black,draw, inner sep=0pt, minimum size=5pt]
\tikzstyle{bobT}=[fill={rgb,255: red,255; green,128; blue,0}, draw=black, inner sep=0pt, shape=circle, minimum size=5pt]

\tikzstyle{dashed grey}=[-, draw={rgb,255: red,128; green,128; blue,128}, dashed]

	\begin{pgfonlayer}{nodelayer}
		\node (1) at (1, 1) {};
		\node (2) at (2, 1) {$\ldots$};
		\node (3) at (3, 1) {};
		\node (4) at (0.5, 1.5) {};
		\node (5) at (0.5, 2) {};
		\node (6) at (1.75, 2) {};
		\node (7) at (2, 1.5) {};
		\node (8) at (3, 1.5) {};
		\node (9) at (3.5, 2) {};
		\node (10) at (4.75, 2) {};
		\node (11) at (4.75, 0.75) {};
		
		\node (13) at (-1, 0.5) {};
		\node (14) at (-1, 2) {};
		\node (15) at (-5.75, 2) {};
		\node (16) at (-6, -1) {};
		\node (17) at (-2, -0.5) {};
		
		\node (19) at (-7, 1) {};
		
		\node (23) at (-1, -1) {};
		\node (24) at (-6, -1.5) {};
		\node (25) at (-2, -1) {};
		\node (26) at (-5.75, -6) {};
		\node (27) at (4.75, -5.75) {};
		\node (28) at (4.75, -3.5) {};
		
		\node (32) at (1, -0.5) {};
		\node (33) at (2, -1.5) {};
		\node (34) at (4.75, -2.75) {};
		\node (35) at (4.75, 0) {};
		\node (36) at (3, 0) {};
		\node (37) at (1, -0.25) {};

        \node [style=rn] (18) at (-4.75, 0.5) [label=below:$A(t_0)$]{};
        \node (20) at (-6.5, 1) {$C_\alpha(\widehat{v})$};
        \node [style=rn] (29) at (-5.25, -5) [label=below:$\widehat{B}$]{};
		\node [style=rn] (30) at (-5, -2.5) [label=below:$B(t_0)$]{};
		\node [style=rn] (38) at (-0.3, -2.1) [label=right:$A(\widetilde{t}-1)$]{};
        \node (21) at (-1.5, -1.5) {};
        \node (12) at (-1, 0.5) {};
        \node [style=rn] (0) at (0, 0) [label=below:{$\widehat{v}$}]{};
		\node [style=rn] (39) at (4, -1) [label=below:$A(\widehat{t})$]{};
		\node (40) at (5.5, -5) {$C_\beta(\widehat{v})$};
		\node (41) at (5.5, -1.5) {$C_\gamma(\widehat{v})$};
        \node [style=rn] (vv) at (-1, -3.5) [label=below:$\widetilde{v}$]{};
        \node [style=rn] (btv) at (2, -4.5) [label=below:$B(\widetilde{t})$]{};
        \node [style=rn] (btv1) at (1.75, -3.75) [label=right:$B(\widetilde{t}-1)$]{};
        \node [color = blue](obs) at (3.75, -5.25) {$C_{\widehat{\delta}}(\widetilde{v})$};
        \node [color = blue](ogs) at (2, -3) {$C_{\widehat{\gamma}}(\widetilde{v})$};
        \node [style=rn] (atc) at (0.2, -2.8 ) [label=below:{$A(t')$}]{};
        \node [color = blue](oalfahat) at (-7.75, -1) {$C_{\widehat{\alpha}}(\widetilde{v})=C_{\widehat{\beta}}(\widetilde{v})$};
        \node [style=rn] (vc) at (-2.75, -2.5) {};
	\end{pgfonlayer}
	\begin{pgfonlayer}{edgelayer}
        \draw (vv)--(-0.75, -3.1); 
        \draw (vv)--(-0.6, -3.3);
        \draw (vv)--(-1, -3);
        \draw[pattern=dots, pattern color=green, dotted] (-7, -1.5) -- (-7.5, 3) -- (5.5, 3) -- (6.25, -1.5) -- (5, -3) -- (-0.75, -1) --(-1.75, -5.75)-- (-5.75, -6) -- cycle;

        \draw[pattern=horizontal lines light blue, pattern color=gray, dotted] (-0.5, -3.75) -- (1, -3.5) -- (4.25, -3.25) -- (4.74, -5.5) -- (1, -5.75) -- cycle;

        \draw[pattern=horizontal lines, pattern color=yellow, dotted] (-0.75, -1.5) -- (-0.5, -3.5) -- (3.5, -3) -- (-0.5, -1.25) -- cycle;
        
        \draw[decorate, thick, color=blue] (vv)--(38);
        \draw[decorate, thick, decoration=zigzag, color=blue] (vc)--(vv);
        \draw[decorate, thick, decoration=zigzag, color=blue] (vv)--(btv);
        \draw[decorate, thick, decoration=zigzag, color=blue] (0) -- (vc.center) -- (29);
        
		\draw (12.center)--(0) to (1.center);
        \draw (0)--(37.center);
		\draw (0) to (3.center);
  
		\draw (1.center) -- (4.center) -- (5.center)-- (6.center)-- (7.center) --cycle;

		\draw (3.center) -- (8.center) -- (9.center) -- (10.center) -- (11.center) -- cycle;

		\draw (14.center) to (13.center);
		\draw (13.center) to (17.center);
		\draw (17.center) to (16.center);
		\draw (16.center) to (15.center);
		\draw (15.center) to (14.center);

		\draw (23.center) to (25.center);
		\draw (25.center) to (24.center);
		\draw (24.center) to (26.center);
		\draw (26.center) to (27.center);
		\draw (27.center) to (28.center);
		\draw (28.center) to (23.center);

		\draw (37.center) to (32.center);
		\draw (32.center) to (33.center);
		\draw (33.center) to (34.center);
		\draw (34.center) to (35.center);
		\draw (35.center) to (36.center);
		\draw (36.center) to (37.center);
	
	\end{pgfonlayer}
\end{tikzpicture}

            \caption{Situation when $C_{\widehat{\alpha}}(\widetilde{v})\not= C_{\widehat{\delta}}(\widetilde{v})$ and $C_{\widehat{\alpha}}(\widetilde{v})= C_{\widehat{\beta}}(\widetilde{v})$, one possibility of Subcase 4.2.}
            \label{fig:case4-2a1}
        \end{figure}

      \begin{figure}[!ht]
            \centering
    \begin{tikzpicture}[scale=0.75]
            \footnotesize
\tikzstyle{rn}=[circle, fill=black,draw, inner sep=0pt, minimum size=5pt]
\tikzstyle{bobT}=[fill={rgb,255: red,255; green,128; blue,0}, draw=black, inner sep=0pt, shape=circle, minimum size=5pt]

\tikzstyle{dashed grey}=[-, draw={rgb,255: red,128; green,128; blue,128}, dashed]

	\begin{pgfonlayer}{nodelayer}
		\node (1) at (1, 1) {};
		\node (2) at (2, 1) {$\ldots$};
		\node (3) at (3, 1) {};
		\node (4) at (0.5, 1.5) {};
		\node (5) at (0.5, 2) {};
		\node (6) at (1.75, 2) {};
		\node (7) at (2, 1.5) {};
		\node (8) at (3, 1.5) {};
		\node (9) at (3.5, 2) {};
		\node (10) at (4.75, 2) {};
		\node (11) at (4.75, 0.75) {};
		
		\node (13) at (-1, 0.5) {};
		\node (14) at (-1, 2) {};
		\node (15) at (-5.75, 2) {};
		\node (16) at (-6, -1) {};
		\node (17) at (-2, -0.5) {};
		
		\node (19) at (-7, 1) {};
		
		\node (23) at (-1, -1) {};
		\node (24) at (-6, -1.5) {};
		\node (25) at (-2, -1) {};
		\node (26) at (-5.75, -6) {};
		\node (27) at (4.75, -5.75) {};
		\node (28) at (4.75, -3.5) {};
		
		\node (32) at (1, -0.5) {};
		\node (33) at (2, -1.5) {};
		\node (34) at (4.75, -2.75) {};
		\node (35) at (4.75, 0) {};
		\node (36) at (3, 0) {};
		\node (37) at (1, -0.25) {};

        \node [style=rn] (18) at (-4.75, 0.5) [label=below:$A(t_0)$]{};
        \node (20) at (-6.5, 1) {$C_\alpha(\widehat{v})$};
        \node [style=rn] (29) at (-5.25, -5) [label=below:$\widehat{B}$]{};
		\node [style=rn] (30) at (-5, -2.5) [label=below:$B(t_0)$]{};
		\node [style=rn] (38) at (-1.25, -3) [label=right:$A(\widetilde{t}-1)$]{};
        \node [style=rn] (atc) at (-0.5, -2 ) [label=right:{$A(t')$}]{};
        \node (21) at (-1.5, -1.5) {};
        \node (12) at (-1, 0.5) {};
        \node [style=rn] (0) at (0, 0) [label=below:{$\widehat{v}$}]{};
		\node [style=rn] (39) at (4, -1) [label=below:$A(\widehat{t})$]{};
		\node (40) at (5.5, -5) {$C_\beta(\widehat{v})$};
		\node (41) at (5.5, -1.5) {$C_\gamma(\widehat{v})$};
        \node [style=rn] (vv) at (-3, -3.5) [label=below:$\widetilde{v}$]{};
        \node [style=rn] (btv) at (2, -4.5) [label=below:$B(\widetilde{t})$]{};
        \node [style=rn] (btv1) at (1.75, -3.75) [label=right:$B(\widetilde{t}-1)$]{};
        \node [color = blue](obs) at (3.75, -5.25) {$C_{\widehat{\delta}}(\widetilde{v})$};
        \node [color = blue](ogs) at (1.5, -2.75) {$C_{\widehat{\gamma}}(\widetilde{v})$};

        \node [color = blue](oalfahat) at (-7.75, -1) {$C_{\widehat{\alpha}}(\widetilde{v})$};
        \node [color = blue](obetahat) at (-3.5, -5.5) {$C_{\widehat{\beta}}(\widetilde{v})$};
	\end{pgfonlayer}
	\begin{pgfonlayer}{edgelayer}
        \draw (vv)--(-3.25, -3.1); 
        \draw (vv)--(-3.4, -3.3);
        \draw (vv)--(-3, -3);
        \draw[pattern=dots, pattern color=green, dotted] (-7, -1.5) -- (-7.5, 3) -- (5.5, 3) -- (6.25, -1.5) -- (5, -3) -- (-0.5, -1) --(-1.75, -2.75)-- (-2.75, -3) -- (-1, -0.5) -- cycle;

        \draw[pattern=horizontal lines light blue, pattern color=gray, dotted] (-0.5, -3.75) -- (1, -3.5) -- (4.25, -3.25) -- (4.74, -5.5) -- (1, -5.75) -- cycle;

        \draw[pattern=horizontal lines, pattern color=yellow, dotted] (-2, -3.25) -- (-0.5, -3.5) -- (3.5, -3) -- (-0.5, -1.25) -- cycle;

        \draw[pattern=dots, pattern color=orange, dotted] (-2, -5) -- (-3, -5.82) -- (-5.5, -5.82) -- (-5.5, -4) -- (-3.5, -3.75) -- cycle;
        
        \draw[decorate, thick, color=blue] (vv)--(38);
        \draw[decorate, thick, decoration=zigzag, color=blue] (vv)--(btv);
        \draw[decorate, thick, decoration=zigzag, color=blue] (0) -- (vv.center) -- (29);
        
		\draw (12.center)--(0) to (1.center);
        \draw (0)--(37.center);
		\draw (0) to (3.center);
  
		\draw (1.center) -- (4.center) -- (5.center)-- (6.center)-- (7.center) --cycle;

		\draw (3.center) -- (8.center) -- (9.center) -- (10.center) -- (11.center) -- cycle;

		\draw (14.center) to (13.center);
		\draw (13.center) to (17.center);
		\draw (17.center) to (16.center);
		\draw (16.center) to (15.center);
		\draw (15.center) to (14.center);

		\draw (23.center) to (25.center);
		\draw (25.center) to (24.center);
		\draw (24.center) to (26.center);
		\draw (26.center) to (27.center);
		\draw (27.center) to (28.center);
		\draw (28.center) to (23.center);

		\draw (37.center) to (32.center);
		\draw (32.center) to (33.center);
		\draw (33.center) to (34.center);
		\draw (34.center) to (35.center);
		\draw (35.center) to (36.center);
		\draw (36.center) to (37.center);
	
	\end{pgfonlayer}
\end{tikzpicture}

            \caption{Situation when $C_{\widehat{\alpha}}(\widetilde{v})\not= C_{\widehat{\delta}}(\widetilde{v})$ and $C_{\widehat{\alpha}}(\widetilde{v})\not= C_{\widehat{\beta}}(\widetilde{v})$, second possibility of Subcase 4.2.}
            \label{fig:case4-2a2}
        \end{figure}
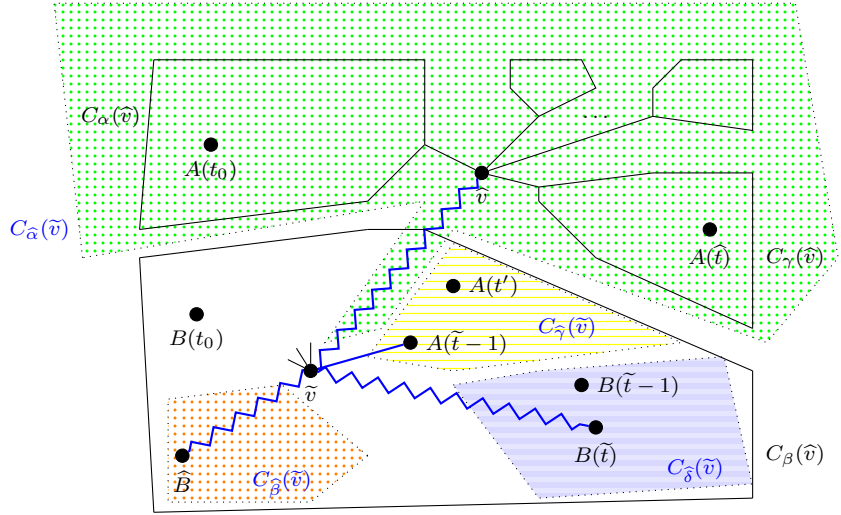

      If $C_{\widehat{\alpha}}(\widetilde{v})\not= C_{\widehat{\delta}}(\widetilde{v})$, then $C_{\widehat{\delta}}(\widetilde{v})\subset C_{\beta}(\widehat{v})$. Two possible situations with respect to $\widetilde{v}$ can occur, they can be seen in Figures \ref{fig:case4-2a1} and \ref{fig:case4-2a2}. Notice that, following the same line of thought as in Subcase 4.1, $\reach(C_{\gamma}(\widehat{v}))\leq \reach(C_{\widehat{\alpha}}(\widetilde{v}))$, $\reach(C_{\widehat{\gamma}}(\widetilde{v}))\leq \reach(C_{\widehat{\beta}}(\widetilde{v}))$ and  $\reach(C_{\widehat{\delta}}(\widetilde{v}))\leq \reach(C_{\widehat{\beta}}(\widetilde{v}))$. 
      Let $t'=\max\{t \mid t<\widetilde{t}-1 \text{ and } B(t)\not \in C_{\widehat{\delta}}(\widetilde{v})\}$, hence $B(t')=\widetilde{v}$. If $A(t')\not \in C_{\widehat{\gamma}}(\widetilde{v})$ then there exists $t'<t''<\widetilde{t}-1$ such that $A(t'')=\widetilde{v}$, but then $A(t'')$ is on the path between $\widehat{B}$ and $B(t'')$, contradicting the definition of $\widetilde{t}$. Therefore $A(t') \in C_{\widehat{\gamma}}(\widetilde{v})$ and $d(A(t'), B(t'))\leq \reach(C_{\widehat{\gamma}}(\widetilde{v}))$.   
      Furthermore, $d(A(\widetilde{t}),B(\widetilde{t}))\leq \reach(C_{\widehat{\delta}}(\widetilde{v}))$. 
      
      Using the obtained inequalities and \eqref{neenakost:BswitchAL21} we obtain 
    \[\min\{d(A(t),B(t))\mid t\in P\} \leq\min\{\reach(C_{\widehat{\alpha}}(\widetilde{v})),\reach(C_{\widehat{\beta}}(\widetilde{v})),\reach(C_{\widehat{\gamma}}(\widetilde{v})),\reach(C_{\widehat{\delta}}(\widetilde{v}))\}\leq \eta(\widetilde{v})\leq \mathfrak{H}(T),\]
      which is a contradiction with the starting hypothesis.

      \begin{figure}[!ht]
            \centering
    \begin{tikzpicture}[scale=0.75]
            \footnotesize
\tikzstyle{rn}=[circle, fill=black,draw, inner sep=0pt, minimum size=5pt]
\tikzstyle{bobT}=[fill={rgb,255: red,255; green,128; blue,0}, draw=black, inner sep=0pt, shape=circle, minimum size=5pt]

\tikzstyle{dashed grey}=[-, draw={rgb,255: red,128; green,128; blue,128}, dashed]

	\begin{pgfonlayer}{nodelayer}
		\node (1) at (1, 1) {};
		\node (2) at (2, 1) {$\ldots$};
		\node (3) at (3, 1) {};
		\node (4) at (0.5, 1.5) {};
		\node (5) at (0.5, 2) {};
		\node (6) at (1.75, 2) {};
		\node (7) at (2, 1.5) {};
		\node (8) at (3, 1.5) {};
		\node (9) at (3.5, 2) {};
		\node (10) at (4.75, 2) {};
		\node (11) at (4.75, 0.75) {};
		
		\node (13) at (-1, 0.5) {};
		\node (14) at (-1, 2) {};
		\node (15) at (-5.75, 2) {};
		\node (16) at (-6, -1) {};
		\node (17) at (-2, -0.5) {};
		
		\node (19) at (-7, 1) {};
		
		\node (23) at (-1, -1) {};
		\node (24) at (-6, -1.5) {};
		\node (25) at (-2, -1) {};
		\node (26) at (-5.75, -6) {};
		\node (27) at (4.75, -5.75) {};
		\node (28) at (4.75, -3.5) {};
		
		\node (32) at (1, -0.5) {};
		\node (33) at (2, -1.5) {};
		\node (34) at (4.75, -2.75) {};
		\node (35) at (4.75, 0) {};
		\node (36) at (3, 0) {};
		\node (37) at (1, -0.25) {};

        \node [style=rn] (18) at (-4.75, 0.5) [label=below:$A(t_0)$]{};
        \node (20) at (-6.5, 1) {$C_\alpha(\widehat{v})$};
        \node [style=rn] (29) at (-5.25, -5) [label=below:$\widehat{B}$]{};
		\node [style=rn] (30) at (-5, -2.5) [label=below:$B(t_0)$]{};
		\node [style=rn] (38) at (-1.25, -3) [label=right:$A(\widetilde{t}-1)$]{};
        \node (21) at (-1.5, -1.5) {};
        \node (12) at (-1, 0.5) {};
        \node [style=rn] (0) at (0, 0) [label=below:{$\widehat{v}$}]{};
		\node [style=rn] (39) at (4, -1) [label=below:$A(\widehat{t})$]{};
		\node (40) at (5.5, -5) {$C_\beta(\widehat{v})$};
		\node (41) at (5.5, -1.5) {$C_\gamma(\widehat{v})$};
        \node [style=rn] (vv) at (-3, -3.5) [label=below:$\widetilde{v}$]{};
        \node [style=rn] (btv) at (-0.5, 1.5) [label=below:$B(\widetilde{t})$]{};
        \node [style=rn] (btv1) at (-0.75, 2.5) [label=right:$B(\widetilde{t}-1)$]{};
        \node [style=rn] (atc) at (-2, -4) [label=below:$A(t')$]{};
        \node [style=rn] (btc) at (-3.75, -4.75) [label=right:$B(t')$]{};
        \node [color = blue](ogs) at (1.5, -2.75) {$C_{\widehat{\gamma}}(\widetilde{v})$};

        \node [color = blue](oalfahat) at (-7.75, -1) {$C_{\widehat{\alpha}}(\widetilde{v})=C_{\widehat{\delta}}(\widetilde{v})$};
        \node [color = blue](obetahat) at (-3.5, -5.5) {$C_{\widehat{\beta}}(\widetilde{v})$};
	\end{pgfonlayer}
	\begin{pgfonlayer}{edgelayer}
        \draw (vv)--(-3.25, -3.1); 
        \draw (vv)--(-3.4, -3.3);
        \draw (vv)--(-3, -3);
        \draw (atc)--(vv);
        \draw[pattern=dots, pattern color=green, dotted] (-7, -1.5) -- (-7.5, 3) -- (5.5, 3) -- (6.25, -1.5) -- (5, -3) -- (-0.5, -1) --(-1.75, -2.75)-- (-2.75, -3) -- (-1, -0.5) -- cycle;

        \draw[pattern=horizontal lines, pattern color=yellow, dotted] (-2, -3.25) -- (-0.5, -3.5) -- (3.5, -3) -- (-0.5, -1.25) -- cycle;

        \draw[pattern=dots, pattern color=orange, dotted] (-2, -5) -- (-3, -5.82) -- (-5.5, -5.82) -- (-5.5, -4) -- (-3.5, -3.75) -- cycle;
        
        \draw[decorate, thick, color=blue] (vv)--(38);
        \draw[decorate, thick, decoration=zigzag, color=blue] (0) -- (vv.center) -- (29);
        
		\draw (12.center)--(0) to (1.center);
        \draw (0)--(37.center);
		\draw (0) to (3.center);
  
		\draw (1.center) -- (4.center) -- (5.center)-- (6.center)-- (7.center) --cycle;

		\draw (3.center) -- (8.center) -- (9.center) -- (10.center) -- (11.center) -- cycle;

		\draw (14.center) to (13.center);
		\draw (13.center) to (17.center);
		\draw (17.center) to (16.center);
		\draw (16.center) to (15.center);
		\draw (15.center) to (14.center);

		\draw (23.center) to (25.center);
		\draw (25.center) to (24.center);
		\draw (24.center) to (26.center);
		\draw (26.center) to (27.center);
		\draw (27.center) to (28.center);
		\draw (28.center) to (23.center);

		\draw (37.center) to (32.center);
		\draw (32.center) to (33.center);
		\draw (33.center) to (34.center);
		\draw (34.center) to (35.center);
		\draw (35.center) to (36.center);
		\draw (36.center) to (37.center);
	
	\end{pgfonlayer}
\end{tikzpicture}

            \caption{Situation when $C_{\widehat{\alpha}}(\widetilde{v})=C_{\widehat{\delta}}(\widetilde{v})$, third possibility of Subcase 4.2.}
            \label{fig:case4-2b}
        \end{figure}
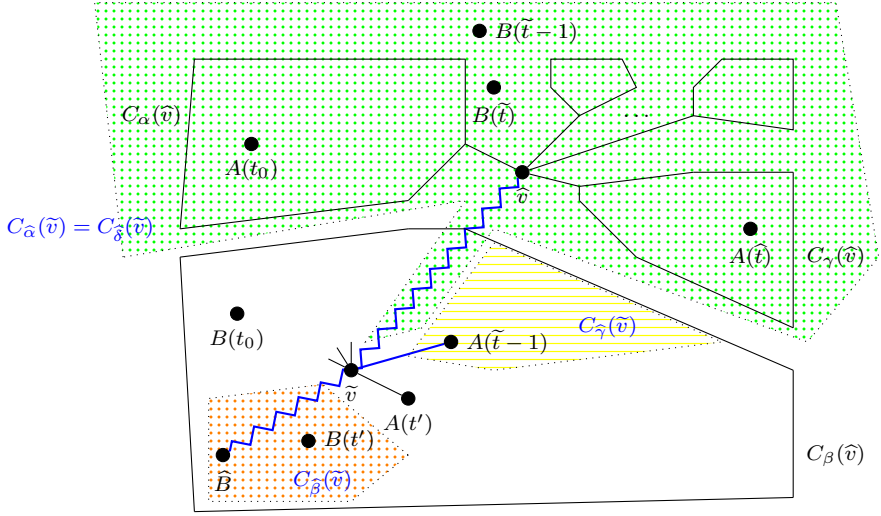
      
       Finally, suppose $C_{\widehat{\alpha}}(\widetilde{v})= C_{\widehat{\delta}}(\widetilde{v})$. Let $t'=\max\{ t\in V(P) \mid t<\widetilde{t}-1 \text{ and } A(t)\not\in \overline{C_{\widehat{\gamma}}(\widetilde{v})}\}$. Then $A(t'+1)=\widetilde{v}$.
       Note that $A(t')$ is adjacent to $\widetilde{v}$. If $B(t')\not \in C_{\widehat{\beta}}(\widetilde{v})$, then  $A(t'+1)$ is in the path from $B(t'+1)$ to $\widehat{B}$, which contradicts the fact that $t'<\widetilde{t}-1$. Hence, $B(t')\in C_{\widehat{\beta}}(\widetilde{v})$. Since $B(\widetilde{t}-1)\in C_{\widehat{\alpha}}(\widetilde{v})$, it follows that 
     there exists $t'<t''<\widetilde{t}-1$ such that $B(t'')=\widetilde{v}$. Hence, $d(A(t''),B(t''))\leq  \reach(C_{\widehat{\gamma}}(\widetilde{v}))$. Using this, \eqref{neenakost:BswitchAL21}, $\reach(C_{\gamma}(\widehat{v}))\leq \reach(C_{\widehat{\alpha}}(\widetilde{v}))$ and  $\reach(C_{\widehat{\gamma}}(\widetilde{v}))\leq \reach(C_{\widehat{\beta}}(\widetilde{v}))$ we obtain
      \[\min\{d(A(t),B(t))\mid t\in P\} \leq \min\{\reach(C_{\widehat{\alpha}}(\widetilde{v})),\reach(C_{\widehat{\beta}}(\widetilde{v})),\reach(C_{\widehat{\gamma}}(\widetilde{v}))\}\leq \eta(\widetilde{v})\leq \mathfrak{H}(T),\]
     which is a contradiction with the starting hypothesis.   
            \end{subcase}
		\end{case} 
        Since we obtained a contradiction in all possible cases, the proof is concluded.
     \end{proof}

\begin{proposition}\label{prop:drevoSpMeja}
    If $T$ is a tree, then $\svSpan{T}\geq \mathfrak{H}(T)$.
\end{proposition}

\begin{proof}
    If $T$ is a path, then by Definition \ref{def:triod-max-eta} $\mathfrak{H}(T)=0$ and by definition of the strong vertex span the assertion holds true.

    Assume $T$ is not a path. Let $v\in V(T)$ be such that $\eta(v) = \mathfrak{H}(T)$, $C_{1}(v), C_{2}(v),...,C_{deg(v)}(v)$ the components of $T-\{v\}$ denoted in such way that $\reach(C_i(v))\geq \reach(C_{i+1}(v))$ for each $i\in \{1, 2, \ldots, deg(v)-1\}$, and $v_j\in C_j(v)$ such that $d(v, v_j)=\eta(v)$ for all $j\in\{1,2,3\}$. Let Alice start her walk in $v_1$ and Bob his in $v_2$. While Alice stays in $v_1$, Bob can visit all the vertices of $T-C_1(v)$ and return to $v_2$ while keeping the distance at least $\eta(v)$ from Alice. Now, Alice moves to $v_3$, while Bob does not move; the distance between them is still at least $\eta(v)$ at all times. Finally, Bob can visit the remaining vertices (i.e. all vertices of $C_1(v)$) an still be at the distance at least $\eta(v)$ from Alice at all times. After Bob visited all vertices of $T$, Alice can use an analogous strategy to visit all the vertices of $T$ while keeping the distance at least $\eta(v)$ from Bob at all times.
\end{proof}

\begin{corollary}\label{spanOfTrees}
    If $T$ is a tree, then 
    \[\svSpan{T}=
    \begin{cases}
        0, & T \text{ is the trivial graph,}\\
        1, & T \text{ is a non-trivial path,} \\
        \mathfrak{H}(T), & \text{otherwise.}
    \end{cases}
    \]
\end{corollary}

\begin{proof}
    The result for the trivial graph and non-trivial paths is presented in \cite{BaTa23}.
    If $T$ is not a path, then it contains a vertex of degree at least 3. Therefore $\mathfrak{H}(T)\geq 1$. If $\mathfrak{H}(T)=1$, then by Proposition \ref{prop:drevoSpMeja} we have that $\svSpan{T}\geq 1$. Supposing $\svSpan{T}\geq 2$ by Theorem \ref{thm:drevoZgMeja} we obtain $\svSpan{T}\leq \mathfrak{H}(T) = 1$, a contradiction. Therefore, if $\mathfrak{H}(T)=1$, then $\svSpan{T}=1$. If $\mathfrak{H}(T)>1$, the result is immediate consequence of Theorem \ref{thm:drevoZgMeja} and Proposition \ref{prop:drevoSpMeja}.
\end{proof}

To state the next result, we need two additional definitions. First, a weak homomorphism $f: G \rightarrow H$ is \emph{edge surjective} if it is surjective and for every $uv \in E(H)$ there exists an edge $xy \in E(G)$ such that $u=f(x)$ and $v=f(y)$. Second, analogous to the definition of the strong vertex span, the strong edge span is defined as follows.

\begin{definition}[\cite{BaTa23}]\label{def:strongESpans}
Let $H$ be a connected graph. Define the \emph{strong edge span} of the graph $H$, denoted by $\seSpan{H}$, as
\[
     \seSpan{H} = \max \{ m_P(f,g)  \mid f,g: P \rightarrow H \text{ are edge surjective weak homomorphisms and $P$ is a path} \}. 
\]
\end{definition}

For trees visiting all vertices equals traversing all edges, hence the strong vertex span and the strong edge span of a tree are equal. Hence, the next corollary is a direct consequence of Corollary \ref{spanOfTrees}.

\begin{corollary}\label{EspanOfTrees}
    If $T$ is a tree, then 
    \[\seSpan{T}=
    \begin{cases}
        0, & T \text{ is the trivial graph,}\\
        1, & T \text{ is a non-trivial path,} \\
        \mathfrak{H}(T), & \text{otherwise.}
    \end{cases}
    \]
\end{corollary}

\section{Computing the strong vertex span of a tree}\label{sec:treesAlg} 

Corollary \ref{spanOfTrees} states that for a tree $T$ different from a path to determine the strong vertex span of $T$ it is necessary to find the triod size of $T$. To do this the notions related to rooted trees are necessary. We follow terminology from the book by Valiente \cite{Va02} for this. Here we restate the crucial notions, for things not explicitly defined here see \cite{Va02}.

Let $T$ be a tree and $r\in V(T)$ a chosen vertex, called \emph{the root of $T$}. The pair $(T, r)$ is called a \emph{rooted tree}. We also say, that $T$ is rooted in $r$. A vertex $u\in V(T)$ is called \emph{the parent} of a vertex $v\in V(T)$, if $u v \in E(T)$ and $d(u,r) < d(v,r)$. In such case, $v$ is called \emph{a child of $u$}. For all $u\in V(T)$, any vertex $v\not= u$ of the subtree rooted in $u$ is called a \emph{descendant} of $u$. For all $u\in V(T)$, the \emph{height} of the subtree rooted in $u$ is the length of a longest path from $u$ to any descendant of $u$. The \emph{center} of $T$ is the set of all vertices of $T$ with eccentricity equal to the radius of $T$, i.e. $\cen(T)=\{u \in V(T) \mid \e(u)=\rad(T)\}$. It is a well-known fact that for every tree $T$ it holds true that $|\cen(T)|=1$ or $|\cen(T)|=2$, moreover if $|\cen(T)|=2$, then the center consists of two adjacent vertices. It is also not difficult to see that the center of a tree can be determined in linear time (e.g. by using two breadth-first-search traversals).

The following lemma gives a nice topological property of all the vertices $u$ of a tree $T$ with the maximum triod size.

\begin{lemma}\label{lem:centersOnPath}
    Let $T$ be a tree and let $S = \{ v \in V(T) \mid \eta(v) = \mathfrak{H}(T) \}$. All vertices of $S$ lie on the same path.
\end{lemma}

\begin{proof}
The case where $|S|\leq 2$ is trivial.  Let $|S|\geq 3$. Working towards a contradiction assume there are three distinct vertices $u, v, w \in S$ that do not lie on the same path. Let $x \in V(T)$ be the unique vertex that lies on the $u,v$-path, $v,w$-path and the $u,w$-path. Clearly, $x\not \in\{u, v, w\}$ otherwise the vertices $u, v, w$ would lie on the same path. It follows that $\eta(x) \geq \min\{\eta(u) + 1, \eta(v) +1, \eta(w) + 1\} \geq \mathfrak{H}(T) +1$, which is a contradiction with the definition of $\mathfrak{H}(T)$.
\end{proof}

Let $T$ be a tree and $v$ its vertex. Using notations from Definition \ref{def:triod-max-eta} for the components $C_i(v)$, for any $i\in\{1,2, \ldots, \deg(v)\}$, and definition of eccentricity of a vertex the following lemma is obvious.

\begin{lemma}\label{lemma:c1-ecc}
    If $T$ is a tree and $v\in V(T)$, then $C_1(v)$ contains $u\in V(T)$ such that $d(u,v)=\e(v)$.
\end{lemma}

Using the notion of rooted trees, the lemmas below provide insight into why it is useful to root a tree in a central vertex for determining the strong vertex span of the tree.

\begin{lemma}\label{lem:C1isAbove}
    Let $T$ be a non-trivial tree rooted in a central vertex of $T$, say $c$. If $v\in V(T)$ is different from $c$, then $C_1(v)$ is the component that contains the parent of $v$.
\end{lemma}

\begin{proof}
    If $v\not= c$ ($v$ is not the root), then $\e(v)\geq\rad(T)$. Note, that $v$ may also be a central vertex, if $T$ is bi-central. In any case, $d(v,c) \geq 1$. Towards a contradiction suppose, that $C_1(v)$ is the component that does not contain the parent of $v$. Therefore, $C_1(v)$ contains descendants of $v$. By Lemma \ref{lemma:c1-ecc} there exists $u \in C_1(v)$ such that $d(u,v) = \e(v)$. But then $d(u,c)=d(u,v)+d(v,c) \geq \rad(T) + 1 > \rad(T)$, which is a contradiction with the fact that $c$ is a central vertex.
\end{proof}

A direct consequence of Lemma \ref{lem:C1isAbove} is the following lemma.

\begin{lemma}\label{lem:C2equalsHeight}
    Let $T$ be a non-trivial tree rooted in a central vertex of $T$, say $c$. If $v\in V(T)$ is different from $c$, then the height of the subtree rooted in $v$ equals $\reach(C_2(v))$, if $v=c$ it equals $\reach(C_1(v))$.
\end{lemma}

Now, we are ready to present an algorithm that computes the strong vertex span of a tree $T$. First, we present a simple depth-first-search (DFS) traversal algorithm that correctly returns the height of the subtree rooted in $v$ of the rooted tree $(T,c)$, where $c$ is a central vertex. This algorithm is presented in Algorithm \ref{alg:height}. Its correctness and complexity is proved in Theorem \ref{thm:algHeight}.

\begin{algorithm}[!ht]
\caption{height($T, c, v$)}
\label{alg:height}
\KwIn{a tree $T$ rooted in a central vertex $c$, a vertex $v$ of the tree $T$ }
\KwOut{height of the subtree rooted in $v$}

label $v$ as \emph{visited}

\tcc{Note, $R_1(v), R_2(v)$ and $R_3(v)$ are global variables initialized before calling this algorithm.}

\ForEach{$u \in N(v)$}
{
\If{$u$ is \emph{not visited}}
{
    h = height(T, c, u)
    
    \If{$h+1 > R_1(v)$}{    \label{line:top3-start}
        $R_3(v) = R_2(v)$
        
        $R_2(v) = R_1(v)$
        
        $R_1(v) = h + 1$
    }
    \ElseIf{$h+1 > R_2(v)$}{
        $R_3(v) = R_2(v)$
        
        $R_2(v)=h+1$
    }
    \ElseIf{$h+1>R_3(v)$}{
        $R_3(v)=h+1$
    }                       \label{line:top3-end}
}
}
\If{ $v\not= c$ }{
\KwRet{$R_2(v)$}
}
\Else{
\KwRet{$R_1(v)$}
}
\end{algorithm}

\begin{theorem}\label{thm:algHeight}
    Given a rooted tree $T$ on $n$ vertices, its root $c$ and a vertex $v \in V(T)$ Algorithm \ref{alg:height} returns the height of the subtree rooted in $v$ of the rooted tree $T$ in $\mathcal{O}(n)$ time. Moreover, for all vertices $v$ of $T$ it also determines $\reach(C_2(v))$ and $\reach(C_3(v))$.
\end{theorem}

\begin{proof}
    Note, $R_1(v), R_2(v)$ and $R_3(v)$ are global variables initialised before calling this algorithm in Algorithm \ref{alg:spanT}.  We use a DFS traversal of the given tree $T$ rooted in a central vertex $c$. For each vertex $v$ we store the two largest reaches amongst the components of $T-\{v\}$ without taking into consideration $C_1(v)$ (notations are used as in Definition \ref{def:triod-max-eta}), i.e. $\reach(C_2(v))$ and $\reach(C_3(v))$ are stored in the values called $R_2(v)$ and $R_3(v)$, respectively. By Lemma \ref{lem:C1isAbove} if $v\not= c$ both $C_2(v)$ and $C_3(v)$ are components rooted in a child of $v$. Lines \ref{line:top3-start} -- \ref{line:top3-end} of Algorithm \ref{alg:height} update these values after receiving the height of the subtree rooted in each child node. The value $R_1(v)$ is used as a guard value, namely if $v=c$ it equals the reach of $C_1(c)$, i.e. $R_1(c)=\reach(C_1(c))=\rad(T)$, by Lemma \ref{lemma:c1-ecc} and the fact that $c$ is a central vertex. If $v\not= c$, then $\e(v)\geq\rad(T)$ and $R_1(v)$ is set to $\rad(T)$ in the initialisation in the base Algorithm \ref{alg:spanT}. Since every subtree (component) of $T-\{v\}$ that does not contain a parent of $v$ has the reach smaller than the radius of $T$ (otherwise $c$ would not be a central vertex), we can use this bound to correctly compute $R_2(v)$ and $R_3(v)$. After the {\bf foreach} loop is finished, the values $R_2(v)$ and $R_3(v)$ equal $\reach(C_2(v))$ and $\reach(C_3(v))$, respectively. 
    If $v\not=c$ then by Lemma \ref{lem:C2equalsHeight} the height of the subtree equals $\reach(C_2(v))$, otherwise it equals $\reach(C_1(v))$. By Lemma \ref{lem:C2equalsHeight} the correct result is returned. 
    
    Since DFS is well known to be linear \cite{Va02} in the size of the graph and our graph is a tree, the proof is concluded.
\end{proof}

\begin{algorithm}[!ht]
\caption{$\svSpan{T}$}
\label{alg:spanT}
\KwIn{a tree $T$}
\KwOut{strong vertex span of $T$}

\If{$|V(T)|=1$} {           \label{line:trivial-start}
    \KwRet{0}
} \ElseIf{$T$ is a path}{
    \KwRet{1}               \label{line:trivial-end}
}

compute the center of $T$ and $\rad(T)$ \label{line:center-radius}

choose $c\in center(T)$

\ForEach{$v \in V(T)$}{ \label{line:init-start}
    label $v$ as not visited

    $R_1(v) =  
    \begin{cases}
        \rad(T), & v \not= c \\
        0, & v = c
    \end{cases}
    $

    $R_2(v) = 0$
    
    $R_3(v) = 0$ \label{line:init-end}
}

height($T, c, c$) \label{line:dfs-height}

$span = 0$      \label{line:find-max-r3-start}

\ForEach{$v \in V(T)$}
{
    \If{$R_3(v)>span$}
    {
        $span = R_3(v)$  \label{line:find-max-r3-end}
    }
}
\KwRet{$span$}
\end{algorithm}

The main algorithm presented in Algorithm \ref{alg:spanT} is used to compute the strong vertex span of an arbitrary tree $T$. If $T$ is the one-vertex graph or a path, then by Corollary \ref{spanOfTrees} the strong vertex span equals 0 or 1, respectively. Checking if a tree is a path can clearly be done in linear time. If $T$ is not a path, then it contains a vertex of degree at least 3. Again, by Corollary \ref{spanOfTrees}, we need to find a vertex $v$ of degree at least three with $\eta(v)=\mathfrak{H}(T)$. Algorithm \ref{alg:spanT} presents how to find such a vertex using the auxiliary Algorithm \ref{alg:height}. 

\begin{theorem}\label{thm:algSpan}
    Algorithm \ref{alg:spanT} returns the strong vertex span of a given tree $T$ on $n$ vertices in $\mathcal{O}(n)$ time. 
\end{theorem}
\begin{proof}
   If $T$ is a path (lines \ref{line:trivial-start} -- \ref{line:trivial-end}), the correctness follows from Corollary \ref{spanOfTrees}. Checking that a tree is a path can clearly be done in linear time. 
   
   Assuming $T$ is not a path, computing the center and the radius of $T$ (line \ref{line:center-radius}) can also be done in linear time. 
   
   The first {\bf foreach} loop (lines \ref{line:init-start} -- \ref{line:init-end}) initialises for each vertex $v$ some global variables needed also in Algorithm \ref{alg:height}, namely: the label of $v$ to not visited, the values $R_1(v), R_2(v)$ and $R_3(v)$ needed to compute bounds or values for the reaches $\reach(C_1(v)), \reach(C_2(v))$ and $\reach(C_3(v))$, respectively. Note, by definition it follows that $R_3(v) = \reach(C_3(v)) = \eta(v)$. The value $R_1(v)$ is set to $\rad(T)$, which is the lower bound for $\reach(C_1(v))$, when $v\not= c$, and equals $\rad(T)$ for $v=c$. Therefore in this case, the starting value is set to 0, so that the height of all subtrees are computed and the three largest ones are stored. This is clearly done in linear time. 
   
   Next, in line \ref{line:dfs-height} the computation of these values is executed by rooting $T$ in the chosen central vertex $c$ by calling Algorithm \ref{alg:height} starting in $c$. By Theorem \ref{thm:algHeight} this is also done in linear time. 
   
   Finally, in the last loop (lines \ref{line:find-max-r3-start} -- \ref{line:find-max-r3-end}) the vertex $v$ with the largest $\eta(v)$ is determined, hence by Corollary \ref{spanOfTrees} and correctness of Algorithm \ref{alg:height} the strong vertex span of $T$ is returned in linear time. 
\end{proof}

Theorem \ref{thm:algSpan} vastly improves the general algorithm for computing the strong vertex span of a graph with the time complexity $\mathcal{O}(n^4)$, see \cite{BaTa23}, to linear time in the case for trees. The following two problems are interesting for further research.

\begin{problem}
    Can the algorithm presented in \cite{BaTa23} for computing the strong vertex span (or other variants) be improved in complexity for the general case?
\end{problem}

\begin{problem}
    For which other families of graphs can the strong vertex span be determined in linear time?
\end{problem}

\section*{Statements and Declarations}

\subsection*{Funding}
Mateja Gra\v si\v c acknowledges the financial support from the Slovenian Research and Innovation Agency (research core funding No. P1-0288). Andrej Taranenko acknowledges the financial support from the Slovenian Research and Innovation Agency (research core funding No. P1-0297 and project N1-0285). All authors acknowledge the financial support from the Slovenian Research and Innovation Agency (project BI-US/22-24-121).



\begin{thebibliography}{99}
\bibitem{BaTa23}
I.~Banič, A.~Taranenko, Span of a Graph: Keeping the Safety Distance, Discrete Mathematics \& Theoretical Computer Science 25:1 (2023) 

\bibitem{Erceg23} Erceg G., Šubašić A., Vojković T. (2023). Some results on the maximal safety distance in a graph. FILOMAT, 37(15), 5123–5136.

\bibitem{Lelek} A.~Lelek, Disjoint mappings and the span of spaces, Fund. Math. 55 (1964), 199 -- 214.

\bibitem{SuVo24} A. Šubašić, T. Vojković, Vertex Spans of Multilayered Cycle and Path Graphs. Axioms 13 (2024) 236. 

\bibitem{Va02} G. Valiente, Algorithms on Trees and Graphs, Springer-Verlag Berlin Heidelberg, 2002.

\end{thebibliography}
\end{document}